\documentclass[10pt]{article}
\usepackage{amsmath,amssymb}
\usepackage{amsthm}
\usepackage[singlelinecheck=off]{caption}
\usepackage{microtype}

\usepackage{subfig}   
\usepackage{xcolor}

\usepackage{graphicx}
\usepackage{algorithm,algpseudocode}
\usepackage{mathrsfs}
\usepackage{epstopdf}
\usepackage{float}
\usepackage{url}
\usepackage{accents}
\DeclareMathOperator{\Tr}{Tr}
\DeclareMathOperator{\rank}{rank}

\newtheorem{assumption}{Assumption}
\newtheorem{proposition}{Proposition}
\newtheorem{lemma}{Lemma}
\newtheorem{corollary}{Corollary}

\allowdisplaybreaks

\usepackage{xparse}

\newlength{\dhatheight}
\newcommand{\doublehat}[1]{%
	\settoheight{\dhatheight}{\ensuremath{\hat{#1}}}%
	\addtolength{\dhatheight}{-0.35ex}%
	\hat{\vphantom{\rule{1pt}{\dhatheight}}%
		\smash{\hat{#1}}}}
	
\title{Efficient construction of tensor ring representations from sampling
\thanks{The work of Y.K. and L.Y. is supported in part by the National Science Foundation under award DMS-1521830 and the U.S. Department of Energy’s Advanced Scientific Computing Research program under award DE-FC02-13ER26134/DE-SC0009409. The work of J.L. is supported in part by the National Science Foundation under award DMS-1454939.}}

\author{Yuehaw Khoo\thanks{Department of Mathematics, Stanford University, Stanford, CA 94305, USA
    (\texttt{ykhoo@stanford.edu}).}
  \and Jianfeng Lu\thanks{Department of Mathematics, Department of Chemistry and Department of
    Physics, Duke University, Durham, NC 27708, USA (\texttt{jianfeng@math.duke.edu}).}
  \and Lexing Ying\thanks{Department of Mathematics and ICME, Stanford University, Stanford, CA
    94305, USA (\texttt{lexing@stanford.edu}).}  }

\begin{document}
	
\maketitle

\begin{abstract}
  In this paper we propose an efficient method to compress a high
  dimensional function into a tensor ring format, based on alternating
  least-squares (ALS). Since the function has size exponential in $d$
  where $d$ is the number of dimensions, we propose efficient sampling
  scheme to obtain $O(d)$ important samples in order to learn the
  tensor ring. Furthermore, we devise an initialization method for ALS
  that allows fast convergence in practice. Numerical examples show
  that to approximate a function with similar accuracy, the tensor
  ring format provided by the proposed method has less parameters than
  tensor-train format and also better respects the structure of the
  original function.
\end{abstract}
	
\section{Introduction}

Consider a function $f:[n]^d \rightarrow \mathbb{R}$ which can be
treated as a tensor of size $n^d$ ($[n]:=\{1,\ldots,n\}$). In order to store and perform algebraic manipulation of the exponentially sized tensor, typically the tensor $f$ has to be decomposed into various low complexity formats. Most current applications involve the CP \cite{hitchcock1927cp} or Tucker decomposition \cite{hitchcock1927cp,tucker1966tucker}. However, the CP decomposition for general tensor is non-unique, whereas the components of Tucker decomposition have exponential size in $d$. The tensor train (TT)  \cite{oseledets2011tensor}, better known as the matrix product states (MPS) proposed earlier in physics literature (see e.g. \cite{AKLT:88,white1992density,perez2006matrix}), emerges as an alternative that breaks the curse of dimensionality while avoiding the ill-posedness issue in tensor decomposition. For this format, function compression and evaluation can be done in $O(d)$ complexity. The situation is however unclear when generalizing a TT to a tensor network. Therefore, in this paper, we consider the compression of a black box function $f$ into a \emph{tensor ring} (TR), i.e., to find 3-tensors $H^1,\ldots,H^d$ such that for $x:=(x_1,\ldots,x_d)\in[n]^d$
\begin{equation}
  \label{TR decomposition}
  f(x_1,\ldots,x_d) \approx \Tr \big( H^1(:,x_1,:)H^2(:,x_2,:)\cdots H^d(:,x_d,:)\big).
\end{equation}
Here $H^k\in \mathbb{R}^{r_{k-1}\times n \times r_{k}}, r_k\leq r$ and we often refer to
$(r_1,\ldots,r_d)$ as the TR rank. Such type of tensor format is a generalization of
the TT format for which $H^1 \in \mathbb{R}^{1\times n \times r_1}, H^d\in \mathbb{R}^{r_{d-1}\times n \times 1}$. The difference between TR and TT is illustrated in Figure~\ref{figure:TTTR} using tensor network diagrams introduced in
Section~\ref{section:notation}. Due to the exponential number of entries, typically we do not have
access to the entire tensor $f$. Therefore, TR format has to be found based on ``interpolation''
from $f(\Omega)$ where $\Omega$ is a subset of $[n]^d$. For simplicity, in the rest of the note, we
assume $r_1=r_2=\ldots=r_d = r$.
\begin{figure}[!ht]
  \begin{center}
	\includegraphics[width = 0.7\columnwidth,trim = 0cm 4cm 0cm
      2cm,clip]{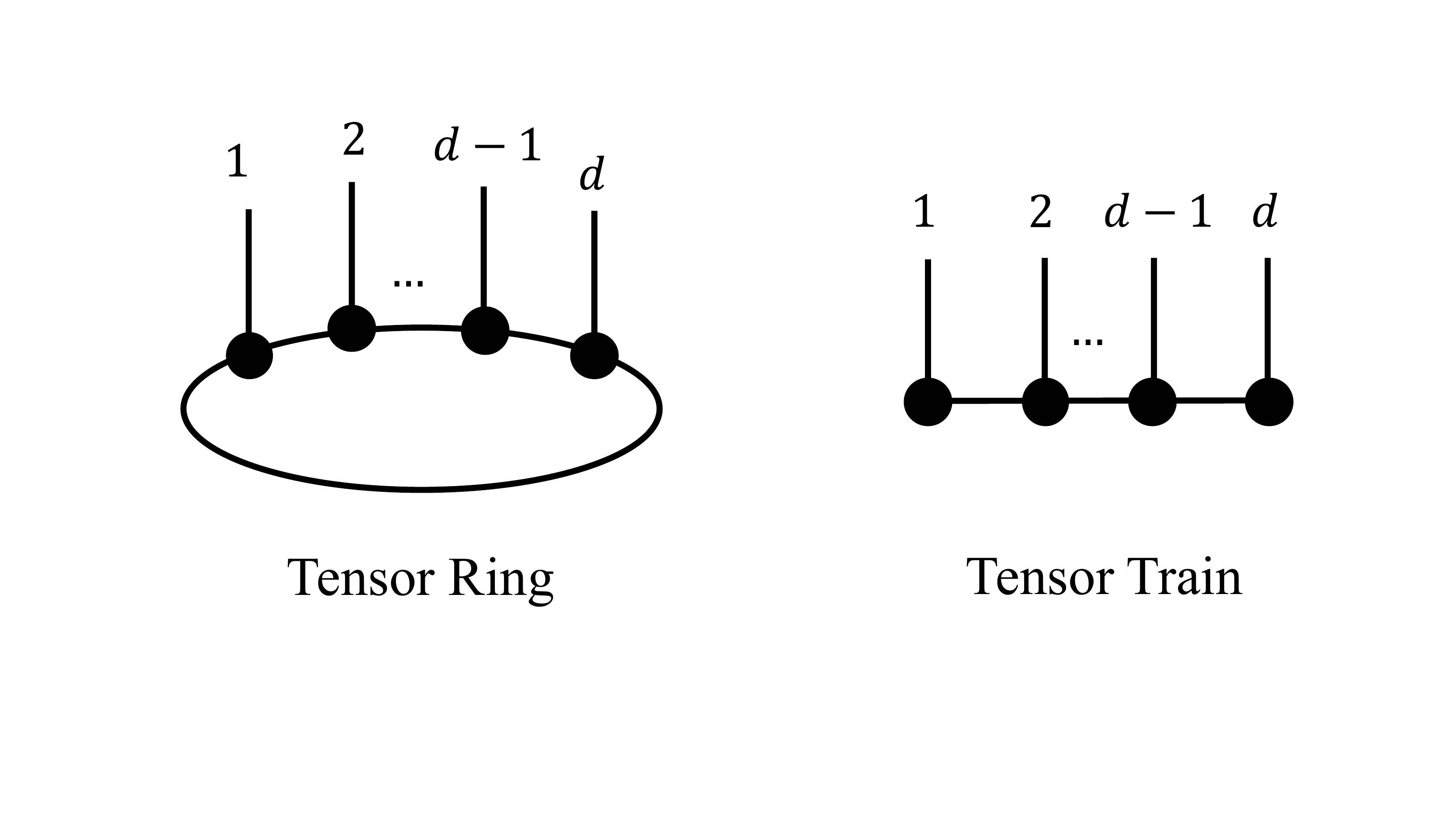}
  \end{center}
  \caption{Comparison between a tensor ring and a tensor train.}\label{figure:TTTR}
\end{figure}

\subsection{Notations}\label{section:notation}
We first summarize the notations used in this note and introduce tensor network diagrams for the
ease of presentation. Depending on the context, $f$ is often referred to as a $d$-tensor of size
$n^d$ (instead of a function). For a $p$-tensor $T$, given two disjoint subsets $\alpha,\beta
\subset [p]$ where $\alpha \cup \beta = [p]$, we use
\begin{equation}
  T_{\alpha; \beta}
\end{equation}
to denote the reshaping of $T$ into a matrix, where the dimensions corresponding to sets $\alpha$
and $\beta$ give rows and columns respectively. Often we need to sample the values of $f$ on a
subset of $[n]^d$ grid points. Let $\alpha$ and $\beta$ be two groups of dimensions where $\alpha
\cup \beta = [d], \alpha\cap \beta = \emptyset$, and $\Omega_1$ and $\Omega_2$ be some subsampled grid points along the subsets of
dimensions $\alpha$ and $\beta$ respectively. We use
\begin{equation}
  f(\Omega_1; \Omega_2):= f_{\alpha;\beta}(\Omega_1 \times \Omega_2)
\end{equation}
to indicate the operation of reshaping $f$ into a matrix, followed by rows and columns subsampling
according to $\Omega_1,\Omega_2$. For any vector $x\in [n]^d$ and any integer $i$, we let
\begin{equation}
  x_i := x_{[(i-1)\bmod d] + 1}.
\end{equation}
For a $p$-tensor $T$, we define its Frobenius norm as
\begin{equation}
  \|T\|_F := \biggl( \sum_{i_1,\ldots,i_p} T(i_1,\ldots,i_p)^2 \biggr)^{1/2}.
\end{equation}
The notation $\text{vec}(A)$ is used to denote the vectorization of a matrix $A$, formed by stacking
the columns of $A$ into a vector. For two sets $\alpha, \beta$, we also use the notation 
\begin{equation}
\alpha\setminus \beta := \{i \in \alpha\ \vert \ i\in \beta^c \}
\end{equation}
to denote the set difference between $\alpha,\beta$.

In this note, for the convenience of presentation, we use tensor network diagrams to represent tensors and contraction between them. A tensor is represented as a node, where the number of legs of a node indicates the dimensionality of the tensor. For example Figure~\ref{figure:tensor} shows a
3-tensor $A$ and a 4-tensor $B$. When joining edges between two tensors (for example in Figure~\ref{figure:contract} we join the third leg of $A$ and first leg of $B$), we mean (with the implicit assumption that the dimensions represented by these legs have the same size)
\begin{equation}
  \sum_{k} A_{i_1 i_2 k} B_{k j_2 j_3 j_4}.
\end{equation}
See the review article \cite{Orus:13} for a more complete introduction of tensor network diagrams.
	
\begin{figure}[!ht]
	  \subfloat[\label{figure:tensor}]{%
			\includegraphics[width = 0.5\columnwidth,trim = 2cm 6cm 2cm 5cm,clip]{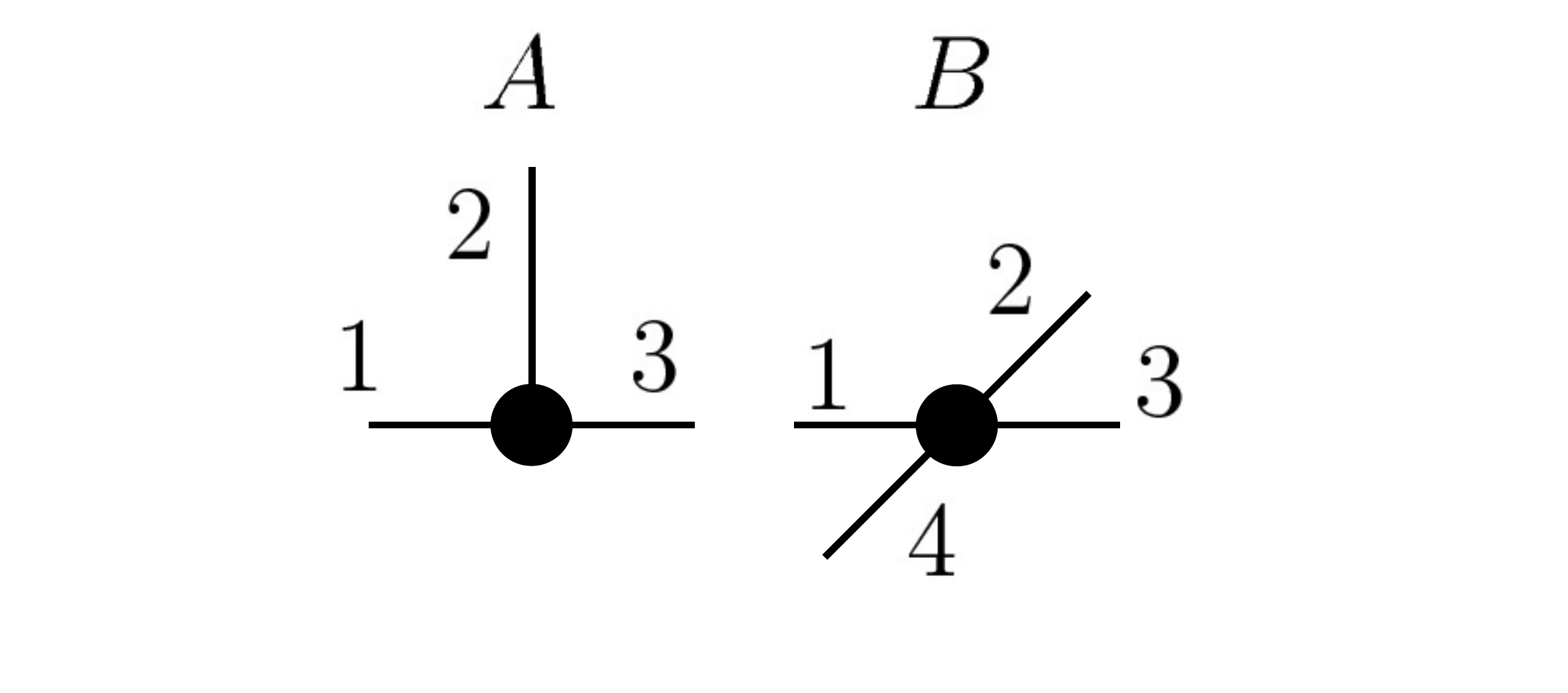}
		}
		\hfill
		\subfloat[\label{figure:contract}]{%
			\includegraphics[width = 0.5\columnwidth,trim = 0cm 7cm 0cm 5cm,clip]{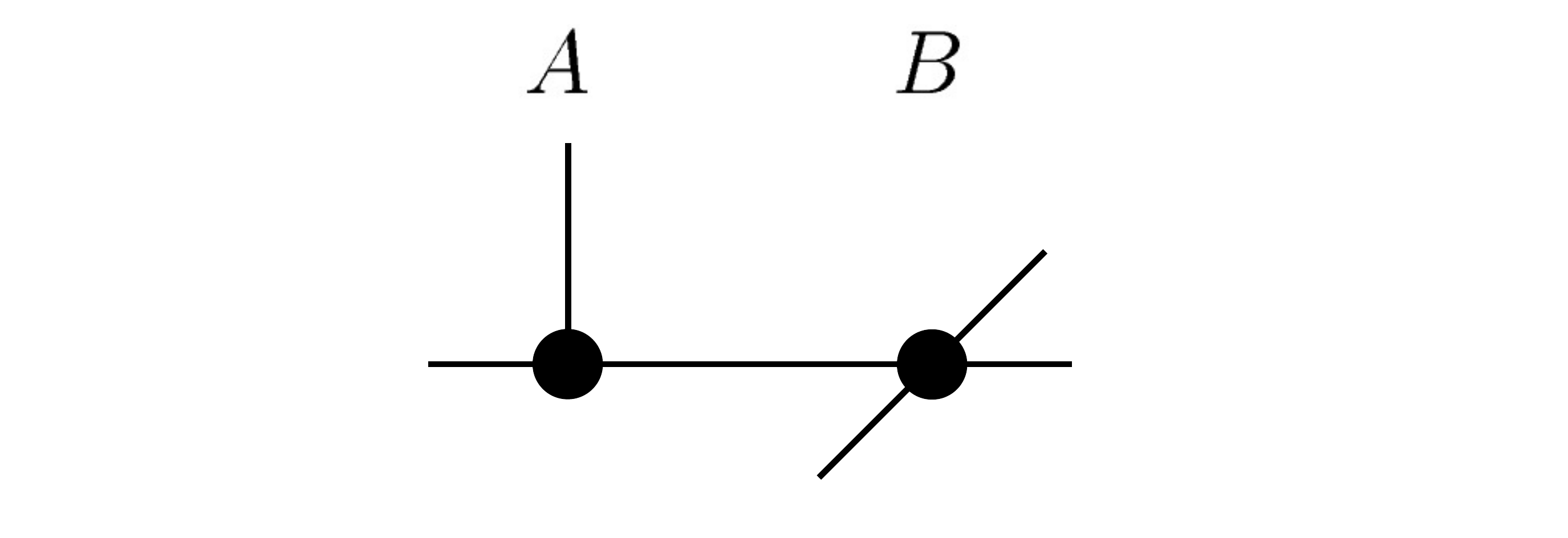}
		}
		\caption{(a) Tensor diagram for a 3-tensor $A$ and a 4-tensor $B$. (b) Contraction between tensors $A$ and $B$.}
\end{figure}

\subsection{Previous approaches}
In this section, we survey previous approaches for compressing a blackbox function into TT or TR. In \cite{oseledets2010tt}, successive CUR (skeleton) decompositions \cite{gantmacher2005applications} are applied to find a decomposition of tensor $f$ in TT format. In
\cite{espig2012note}, a similar scheme is applied to find a TR decomposition of the tensor. A
crucial step in \cite{espig2012note} is to ``disentangle'' one of the 3-tensors $H^k$'s, say $H^1$,
from the tensor ring. First, $f$ is treated as a matrix where the first dimension of $f$ gives rows,
the $2$-nd, $3$-rd, $\ldots,d$-th dimensions of $f$ give columns, i.e., reshaping $f$ to
$f_{1;[d]\setminus 1}$. Then CUR decomposition is applied such that
\begin{equation}
	  f_{1;[d]\setminus 1} = C U R
\end{equation}
and the matrix $C\in \mathbb{R}^{n\times r^2}$ in the decomposition is regarded as $H^1_{2;3,1}$
(the $R$ part in CUR decomposition is never formed due to its exponential size). As noted by the
authors in \cite{espig2012note}, a shortcoming of the method lies in the reshaping of $C$ into
$H^1$. As in any factorization of low-rank matrix, there is an inherent ambiguity for CUR
decomposition in that $CUR = CA A^{-1} UR$ for any invertible matrix $A$. Such ambiguity in
determining $H^1$ may lead to large tensor-ring rank in the subsequent determination of
$H^2,H^3\ldots,H^d$. More recently, \cite{zhao2016tensor} proposes various ALS-based techniques to
determine the TR decomposition of a tensor $f$. However, they only consider the situation where
entries of $f$ are fully observed, which limits the applicability of their algorithms to the case
with rather small $d$. Moreover, depends on the initialization, ALS can suffer from slow convergence. In \cite{wang2017efficient}, ALS is used to determine the TR in a more general setting where only partial observations of the function $f$ are given. In this paper,  we further assume the freedom to observe any $O(d)$ entries from the tensor $f$. As we shall see, leveraging such freedom, the complexity of the iterations can be reduced significantly compare to the ALS procedure in \cite{wang2017efficient}.

\subsection{Our contributions}
In this paper, assuming $f$ admits a rank-$r$ TR decomposition, we propose an ALS-based two-phase
method to reconstruct the TR when only a few entries of $f$ can be sampled. Here we summarize our contributions.
\begin{enumerate}
\item  The optimization problem of finding the TR decomposition is non-convex
  hence requires good initialization in general. We devise method for initializing $H^1,\ldots,H^d$
  that helps to resolve the aforementioned ambiguity issues via certain probabilistic assumption on the
  function $f$.
\item When updating each 3-tensors in the TR, it is infeasible to use all the entries of $f$. We
  devise a hierarchical strategy to choose the samples of $f$ efficiently via interpolative
  decomposition. Furthermore, the samples are chosen in a way that makes the per iteration complexity of the ALS linear in $d$. 
\end{enumerate}

While we focus in this note the problem of construction tensor ring format, the above proposed
strategies can be applied to tensor networks in higher spatial configuration (like PEPS, see e.g.,
\cite{Orus:13}), which will be considered in future works. 

The paper is organized as followed. In Section \ref{section:method} we detail the proposed
algorithm. In Section \ref{section:Motivation}, we provide intuition and theoretical guarantess to motivate the proposed initialization procedure, based on certain probabilistic assumption on $f$. In Section \ref{section:numerical}, we demonstrate the effectiveness of our methods through numerical examples. Finally we conclude the paper in Section \ref{section:conclusion}. 

\section{Proposed method} \label{section:method}

In order to find a tensor ring decomposition (\ref{TR decomposition}), our overall strategy is to
solve the minimization problem
\begin{equation}
	  \label{full variational}
	\min_{H^1,\ldots,H^d} \sum_{x\in [n]^d} \big(\Tr(H^1[x_1]\cdots H^d[x_d])- f(x_1,\ldots,x_d)\big)^2
\end{equation}
where
\[
H^k[x_k] := H^k(:,x_k,:)\in \mathbb{R}^{r\times r}
\]
denotes the $x_k$-th slice of the 3-tensor $H^k$ along the second dimension. It is computationally
infeasible just to set up problem (\ref{full variational}), as we need to evaluate $f$ $n^d$
times. Therefore, analogous to the matrix or CP-tensor completion problem \cite{candes2009exact,yuan2016tensor}, a ``tensor ring completion'' problem \cite{wang2017efficient}
\begin{equation}
	\label{variational}
	\min_{H^1,\ldots,H^d} \sum_{x \in \Omega} \big( \Tr(H^1[x_1]\cdots H^d[x_d])- f(x_1,\ldots,x_d) \big)^2
\end{equation}
where $\Omega$ is a subset of $[n]^d$ should be solved instead. Since there are a total of $dnr^2$
parameters for the tensors $H^1,\ldots,H^d$, there is hope that by observing a small number of
entries in $f$ (at least $O(ndr^2)$), we can obtain the rank-$r$ TR.
	
A standard approach for solving the minimization problem of the type (\ref{variational}) is via
alternating least-squares (ALS). At every iteration of ALS, a particular $H^k$ is treated as
variable while $H^l, \l\neq k$ are kept fixed. Then $H^k$ is optimized w.r.t. the least-squares cost
in (\ref{variational}). More precisely, to determine $H^k$, we solve
\begin{equation}
  \label{k-th least squares}
	\min_{H^k} \sum_{x\in \Omega} \bigl(\Tr(H^k[x_k] C^{x\setminus x_k} )- f(x)\bigr)^2,
\end{equation}
where each coefficient matrix
\begin{equation}
	\label{coefficient matrix}
	C^{x\setminus x_k}  :=  H^{k+1}[x_{k+1}]\ldots H^{d}[x_{d}] H^{1}[x_{1}]\ldots H^{k-1}[x_{k-1}],\quad x\in \Omega.
\end{equation}
By an abuse of notation, we use $x\setminus x_k$ to denote the exclusion of $x_k$ from the $d$-tuple $x$. As mentioned previously, $\vert \Omega \vert$ should be at least $O(ndr^2)$ in order to determine the tensor ring decomposition. This creates a large
computational cost in each iteration of the ALS, as it takes $\vert \Omega \vert (d-1)$ (which has $O(d^2)$ scaling as $\vert \Omega \vert$ has size $O(d)$) matrix
multiplications just to construct $C^{x\setminus x_k}$ for all $x\in \Omega$. When $d$ is large, such quadratic scaling in $d$ for setting up the least-squares problem in each iteration of the ALS is undesirable. 

The following simple but crucial observation allows us to gain a further speed up. Although $O(ndr^2)$ observations of $f$ are required to determine all the components $H^1,\ldots,H^d$,  when it comes to determining each individual $H^k$ via solving the linear system \eqref{k-th least squares}, only $O(nr^2)$ equations are required for the well-posedness of the linear system. This motivates us to use different $\Omega_k$'s each having size $O(nr^2)$ (with $\vert \Omega_1 \vert + \cdots  + \vert \Omega_k \vert \sim O(ndr^2)$ ) to determine different
$H^k$'s in the ALS steps instead of using a fixed set $\Omega$ with size $O(ndr^2)$ for $H^k$'s. If $\Omega_k$ is constructed from
densely sampling the dimensions near $k$ (where neighborhood is defined according to ring geometry)
while sparsely sampling the dimensions far away from $k$, computational savings can be achieved. The
specific construction of $\Omega_k$ is made precise in Section \ref{section: construct envi}. We
further remark that if
\begin{equation}
  \Tr(H^k[x_k] C^{x\setminus x_k} ) \approx f(x)
\end{equation}
holds with small error for every $x\in [n]^d$, then using any $\Omega_k\in [n]^d$ in place
of $\Omega$ in (\ref{k-th least squares}) should give similar solutions, as long as (\ref{k-th least
  squares}) is well-posed. Therefore, we solve
\begin{equation}
	\label{k-th least squares indep}
	\min_{H^k} \sum_{x\in \Omega_k} \bigl(\Tr(H^k[x_k] C^{x\setminus x_k} )- f(x)\bigr)^2
\end{equation}
instead of \eqref{k-th least squares} in each step of the ALS where the index sets $\Omega_k$'s depend on $k$. We note that
in practice, a regularization term $\lambda \sigma_k \|H^k(x_k) \|_F^2$ is added to the cost in
(\ref{k-th least squares indep}) to reduce numerical instability resulting from potential high
condition number of the least-squares problem (\ref{k-th least squares indep}). In all of our
experiments, $\lambda$ is set to $10^{-9}$ and $\sigma_k$ is the top singular values of the Hessian
of the least-squares problem (\ref{k-th least squares indep}). The quality of TR is rather
insensitive to the choice of $\lambda$ as long as the value is kept small.
	
At this point it is clear that there are two issues needed to be addressed. The first issue is
concerning the choice of $\Omega_k, k\in[d]$. Another issue is that non-convex nature of the tensor ring completion problem \ref{variational} may cause ALS to converge to a poor local minima. We solve the first issue using a hierarchical sampling strategy. As for the second issue, by making certain probabilistic assumption on $f$, we are able to obtain cheap and
intuitive initialization that allows fast convergence. Before moving on, we summarize the full algorithm in Algorithm~\ref{alg:
  algorithm1}. The steps of Algorithm~\ref{alg: algorithm1} are further detailed in Section \ref{section: construct envi}, \ref{section: initialization}, and \ref{section:ALS}.
\begin{algorithm}[ht]
  \caption{Alternating least squares}\label{alg: algorithm1}
  \begin{algorithmic}[1]
	\Require 
    \Statex Function $f:[n]^d\rightarrow \mathbb{R}$.
	\Ensure 
    \Statex Tensor ring $H^1,\ldots,H^d\in\mathbb{R}^{r\times n \times r}$.
	\vspace{1mm}
	\State Identify the index sets $\Omega_k$'s and compute $f(\Omega_k)$ for each $k\in[d]$ (Section~\ref{section: construct envi}).
	\State Initialize $H^1,\ldots,H^d$ (Section~\ref{section: initialization}).
	\State Start ALS by solving (\ref{k-th least squares indep}) for each $k\in[d]$ (Section~\ref{section:ALS}).
  \end{algorithmic}
\end{algorithm}

\subsection{Constructing $\Omega_k$}
\label{section: construct envi}
In this section, we detail the construction of $\Omega_k$ for each $k\in[d]$. We first construct an
index set $\Omega^\text{envi}_k \subset [n]^{d-3}$ with fixed size $s$. The elements in
$\Omega^\text{envi}_k$ corresponds to different choices of indices for the $[d]\setminus
\{k-1,k,k+1\}$-th dimensions of the function $f$. Then for each of the elements in
$\Omega^\text{envi}_k$, we sample all possible indices from the $(k-1)$-th, $k$-th, $(k+1)$-th
dimensions of $f$ to construct $\Omega_k$, i.e., letting
\begin{equation}
  \label{nodewise sample}
  \Omega_{k} = [n]^3\times \Omega^\text{envi}_k.
\end{equation}
We let $\vert \Omega^\text{envi}_k \vert = s$ for all $k$ where $s$ is a constant that does not depend of the dimension $d$.  In this case, when determining $C^{x\setminus x_k}, x\in \Omega_{k}$ in (\ref{k-th least squares indep}), only
$O(\vert \Omega^\text{envi}_k \vert d)$ multiplications of $r\times r$ matrices are needed, giving a complexity that is linear in $d$ when setting up the least-squares problem. We want to emphasize that although naively it seems that $O(n^3)$ samples are needed to construct $\Omega_k$ in \eqref{nodewise sample}, the $n^3$ samples corresponding to each sample in  $\Omega^\text{envi}_k$ can be obtained via applying interpolative decomposition \cite{friedland2011fast} to the $n\times n \times n$ tensor with $O(n)$ observations. 

It remains that $\Omega^\text{envi}_k$'s need to be constructed. There are two criteria we use for
constructing $\Omega^\text{envi}_k, k\in[d]$. First, we want the range of $f_{k;[d]\setminus
  k}(\Omega_k)$ to be the same as the range of $f_{k;[d]\setminus k}$. Since we expect 
\begin{equation}
\label{least-squares indep reshaped}
H^k_{2;3,1} [\text{vec}(C^{x\setminus x_k})]_{x\in \Omega_k} \approx f_{k;[d]\setminus k}(\Omega_k)
\end{equation}
is enforced through the least-squares in \eqref{k-th least squares indep}, the range of $H^k_{2;3,1}$ is similar to the range of $f_{k;[d]\setminus k}(\Omega_k)$. On the other hand, as we expect the optimal $H^k$ to satisfy 
\begin{equation}
H^k_{2;3,1} [\text{vec}(C^{x\setminus x_k})]_{x\in [n]^d} = f_{k;[d]\setminus k},
\end{equation}
for all the entries of $f$,  then
\begin{equation}
\label{full least-squares reshaped}
\text{Range}(f_{k;[d]\setminus k}(\Omega_k)) = \text{Range}( f_{k;[d]\setminus k}).
\end{equation}
Eq. \eqref{least-squares indep reshaped} and \eqref{full least-squares reshaped} lead us to require that 
\begin{equation}
\text{Range}(f_{k;[d]\setminus k}(\Omega_k)) = \text{Range}(f_{k;[d]\setminus k}(\Omega_k)).
\end{equation}
Here we emphasize that it
is possible to reshape $f(\Omega_k)$ into a matrix $f_{k;[d]\setminus k}(\Omega_k)$ as in \eqref{least-squares indep reshaped} due to the
product structure of $\Omega_k$ in (\ref{nodewise sample}), where the indices along dimension $k$
are fully sampled. The second criteria is that we require the cost in (\ref{k-th least squares
  indep}) to approximate the cost in (\ref{full variational}).

To meet the first criteria, we propose a hierarchical strategy to determine $\Omega^\text{envi}_k$
such that $f_{k;[d]\setminus k}(\Omega_k)$ has large singular values. Assuming $d=3\cdot 2^L$ for
some natural number $L$, we summarize such strategy in Algorithm~\ref{alg:upward pass} (the upward pass)
and \ref{alg:downward pass} (the downward pass). The dimensions are divided into groups of size $3
\cdot 2^{L-l}$ on each level $l$ for $l = 1, \ldots, L$. We emphasize that level $l=1$ corresponds to the coarsest partitioning of the dimensions of the tensor $f$. The purpose of the upward pass is to
hierarchically find \emph{skeletons} $\Theta_{k}^{\text{in}, l}$ which represent the $k$-th group of
indices, while the downward pass hierarchically constructs representative environment skeletons
$\Theta_k^{\text{envi}, l}$. At each level, the skeletons are found by using rank revealing QR (RRQR)
factorization \cite{hong1992rank}.
\begin{algorithm}[H]
  \caption{Upward pass }\label{alg:upward pass}
  \begin{algorithmic}[1]
	\Require
	\Statex Function $f:[n]^d\rightarrow \mathbb{R}$, number of skeletons $s$.
	\Ensure
	\Statex Skeleton sets $\Theta^{\text{in},l}_k$'s
	\State Decimate the number of dimensions by clustering every three dimensions. More precisely, for each $k\in[2^L]$, let
	\[
    \tilde \Theta_k^{\text{in},L} := \{(x_{3k-2},x_{3k-1},x_{3k})\ \vert\ x_{3k-2},x_{3k-1},x_{3k} \in [n] \}.
    \]
	There are $2^L$ index-sets after this step. For each $k\in[2^L]$, construct the set of environment \emph{skeletons}
	\begin{equation}
	  \Theta_k^{\text{envi},l} \subset [n]^{d-3},
	\end{equation}
	with $s$ elements either by selecting multi-indices from
    $[n]^{d-3}$ randomly, or by using the output of
    Algorithm~\ref{alg:downward pass} (when an iteration of upward
    and downward passes is employed). This step is illustrated
    in the following figure.
	\begin{center}
	  \includegraphics[width = 0.8\columnwidth]{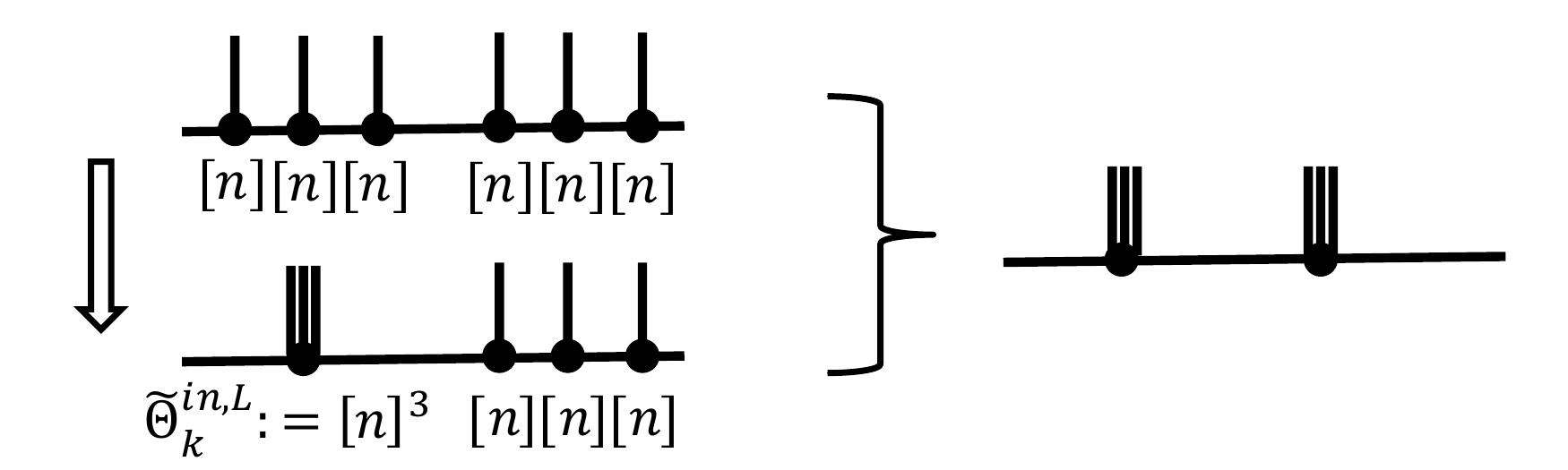}
	\end{center}
      
      (continued on page~\pageref{alg:upward
        pass2}.)
    
	\algstore{upward}
  \end{algorithmic}
\end{algorithm}

\begin{algorithm}[H]
  \begin{algorithmic}\label{alg:upward pass2}
	\algrestore{upward}
	\Statex \hspace{\algorithmicindent} (continued from Algorithm~\ref{alg:upward pass}.)
	\Statex  \hspace{\algorithmicindent}

	\Statex \textbf{for} {$l=L$ to $l=1$}
	\State
    \hspace{\algorithmicindent} \parbox[t]{\dimexpr\linewidth-\algorithmicindent}{Find the skeletons
      within each index-set $\tilde \Theta_{k}^{\text{in},l}$, $k\in[2^l]$ where the elements in
      each $\tilde \Theta_{k}^{\text{in},l}$ are multi-indices of length $3\cdot 2^{L-l}$. Apply RRQR factorization to the
      matrix
	  \begin{equation}
		f( \Theta_k^{\text{envi},l}; \tilde \Theta_k^{\text{in},l}) \in \mathbb{R}^{s\times \vert \tilde \Theta_k^{\text{in},l} \vert }
	  \end{equation}
      to select $s$ columns that best resembles the range of $f( \Theta_k^{\text{envi},l};
      \tilde \Theta_k^{\text{in},l}) $. The multi-indices for these $s$ columns form the set
      $\Theta_k^{\text{in},l}$. Store $\Theta_k^{\text{in},l}$ for each $k\in[2^l]$. This step
      is illustrated in the following figure, where the thick lines are used to denote the
      index-sets with size larger than $s$. 
    \begin{center}
	  \includegraphics[width = 0.8\columnwidth]{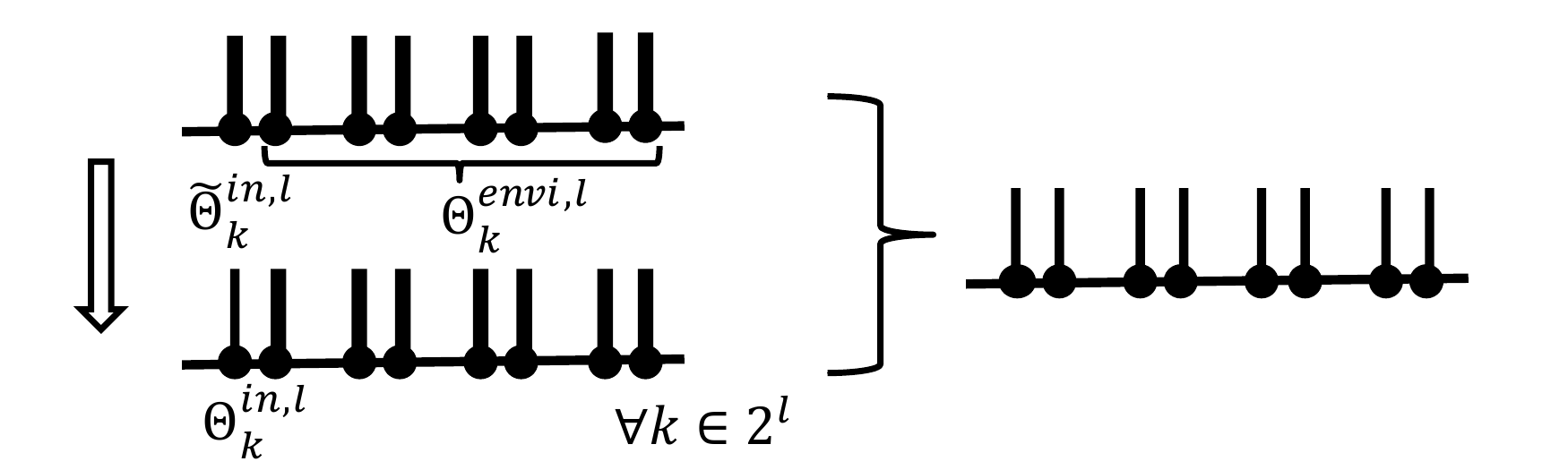}
	\end{center}\strut}

	\State \hspace{\algorithmicindent}  \parbox[t]{\dimexpr\linewidth-\algorithmicindent}{If $l>1$, for each $k\in [2^{l-1}]$, construct
	  \begin{equation}
		\tilde \Theta_{k}^{\text{in},l-1} := \Theta_{2k-1}^{\text{in},l} \times \Theta_{2k}^{\text{in},l}. 
	  \end{equation}
	  Then, sample $s$ elements randomly from
	  \begin{equation}
		\prod_{j\in [2^{l}]\setminus \{2k-1,2k\}} \Theta_j^{\text{in},l}
	  \end{equation} 
	  to form $\Theta_k^{\text{envi},l-1}$, or by using the output of Algorithm~\ref{alg:downward pass}
      (when an iteration of upward and downward passes  is employed). This step is depicted in the
      next figure, and again thick lines are used to denote the index-sets with size larger than
      $d$.\strut}
	\begin{center}
	  \includegraphics[width = 0.8\columnwidth]{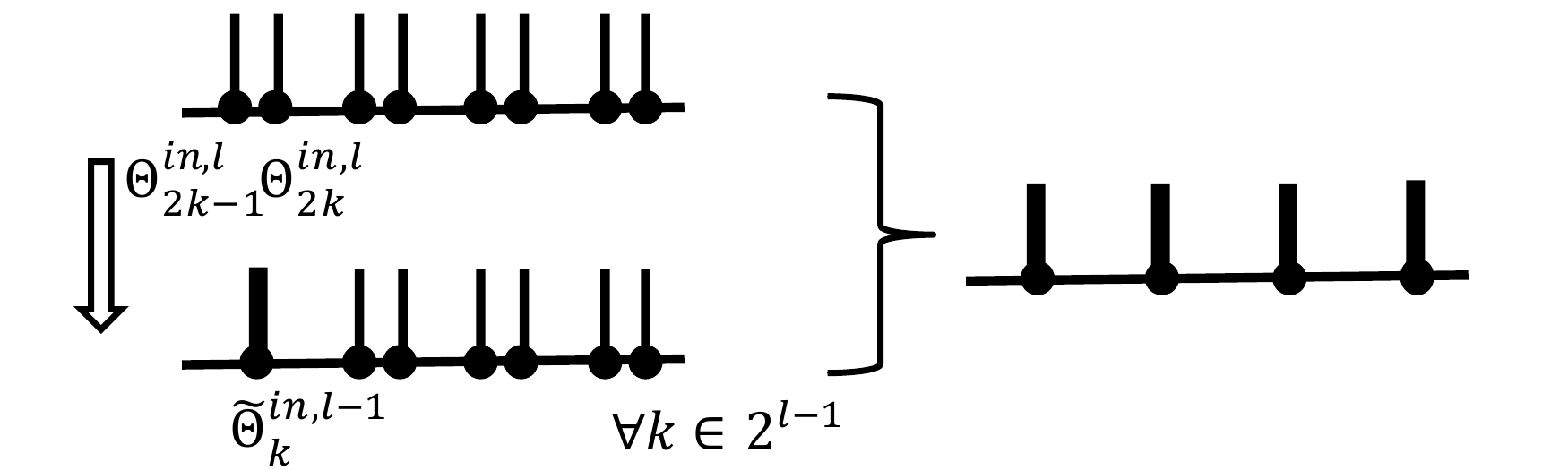}
	\end{center}
    \Statex \hspace{\algorithmicindent} \textbf{end for}
  \end{algorithmic}
\end{algorithm}

After a full upward-downward pass where RRQR are called
$O(d\log d)$ times, $\Theta_k^{\text{envi},L}$ with $k\in[2^L]$ are obtained. Then another upward
pass can be re-initiated. Instead of sampling new $\Theta_k^{\text{envi},l}$'s, the stored
$\Theta_k^{\text{envi},l}$'s in the downward pass are used. Multiple upward-downward passes can be
called to further improved these skeletons. Finally, we let
\begin{equation}
  \Omega_{3k-1}^\text{envi} := \Theta_k^\text{envi},\quad k\in[2^L].
\end{equation}
Observe that we have only obtained $\Omega_k^\text{envi}$ for $k=2,5,\ldots,d-1$. Therefore, we need
to apply upward-downward pass to different groupings of tensor $f$'s dimensions in step (1) of the
upward pass. More precisely, we group the dimensions as $(2, 3, 4), (5,6,7),\ldots,(d-1,d,1)$ and
$(d,1,2), (3,4,5),\ldots, (d-3,d-2,d-1)$ when initializing the upward pass to determine
$\Omega_k^\text{envi}$ with $k=3,6,\ldots,d$ and $k=1,4,\ldots,d-2$ respectively.


\begin{algorithm}[H]
  
  \caption{Downward pass }\label{alg:downward pass}
  \begin{algorithmic}[1]
	\Require
	\Statex Function $f:[n]^d\rightarrow \mathbb{R}$, $\Theta^{\text{in},l}_k$'s from the upward pass, number of skeletons $s$.
	\Ensure
	\Statex Skeletons $\Theta^{\text{envi},l}_k$'s
	\State Let $\Theta_1^{\text{envi},1} = \Theta_2^{\text{in},1}$,  $\Theta_2^{\text{envi},1} = \Theta_1^{\text{in},1}$.
	\begin{center}
	  \includegraphics[width = 0.8\columnwidth]{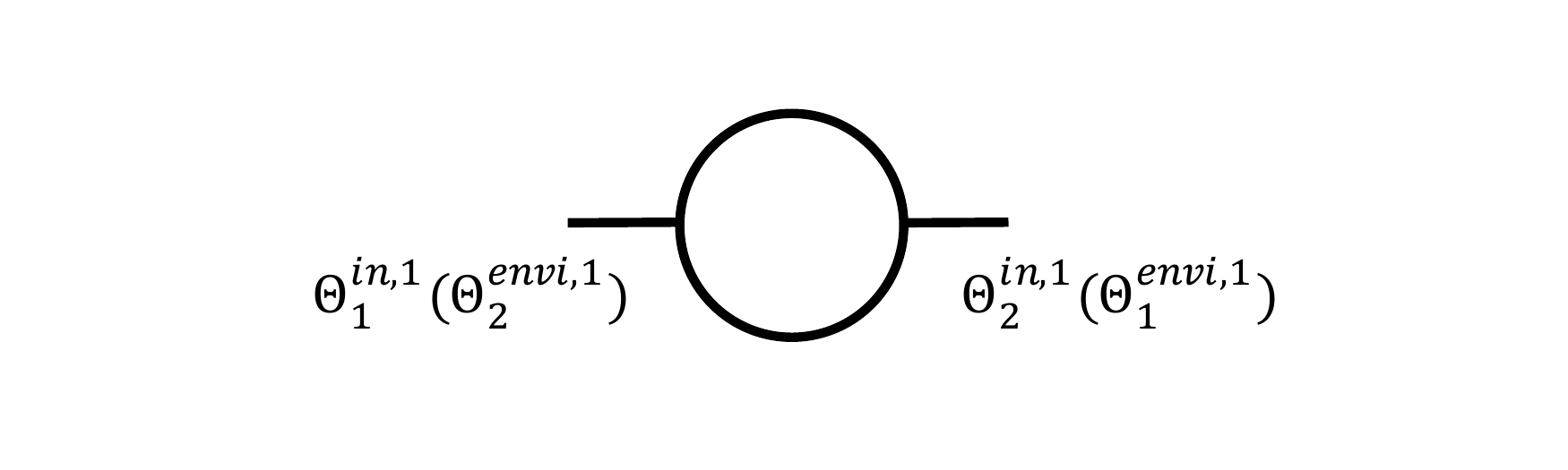}
	\end{center}
	\Statex \textbf{for} {$l=2$ to $l=L$}
    \State \hspace{\algorithmicindent} \parbox[t]{\dimexpr\linewidth-\algorithmicindent}{For each
      $k\in [2^l]$, we obtain $\Theta_k^{\text{envi},l}$ by applying RRQR factorization to
	  \begin{equation}
		f(\Theta_k^{\text{in},l}; \Theta_{k+1}^{\text{in},l}\times \Theta_{(k+1)/2}^{\text{envi} , l-1})
	  \end{equation}
	  or
	  \begin{equation}
		f(\Theta_k^{\text{in},l}; \Theta_{k-1}^{\text{in},l}\times \Theta_{k/2}^{\text{envi} , l-1})
	  \end{equation}
	  for odd or even $k$ respectively to obtain $s$ important columns. The multi-indices
      corresponding to these $s$ columns are used to update $\Theta_k^{\text{envi},l}$. The
      selection of the environment skeletons when $k$ is odd is illustrated in the next
      figure. \strut}
	\begin{center}
	  \includegraphics[width = 0.8\columnwidth]{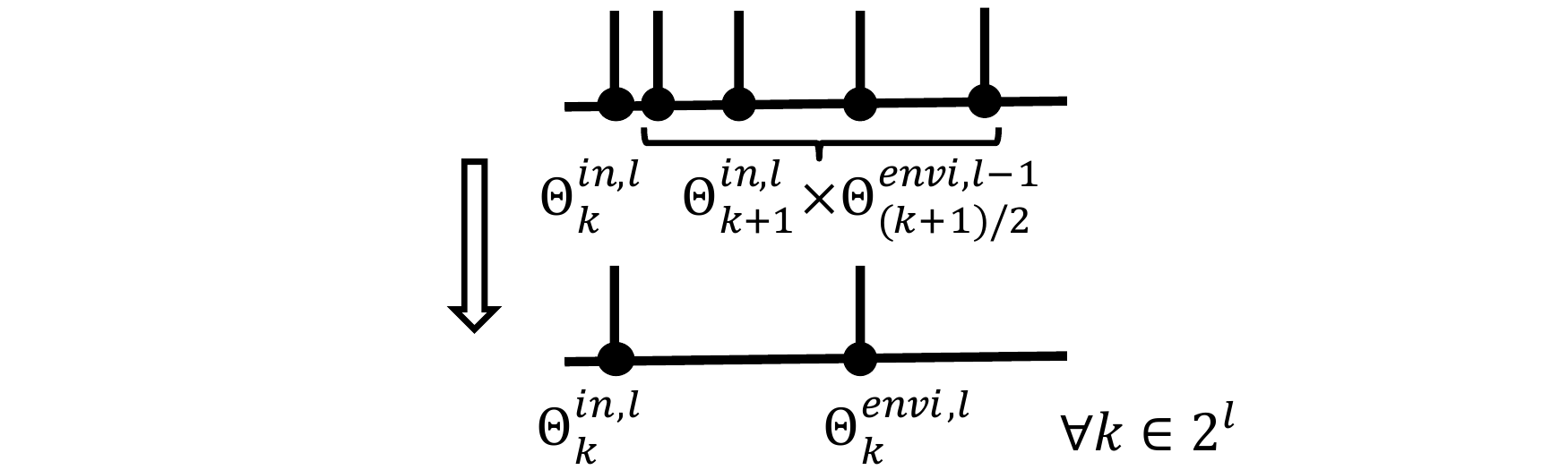}
	\end{center}
    \Statex \textbf{end for}
  \end{algorithmic}
\end{algorithm}
	

Finally, to meet the second criteria that the cost in (\ref{k-th least squares indep}) should
approximate the cost in (\ref{full variational}), to each $\Omega_k^\text{envi}$, we add extra
samples $x\in [n]^{d-3}$ by sampling $x_i$'s uniformly and independently from $[n]$. We typically
sample an extra $5s$ samples to each $\Omega_k^\text{envi}$.  This completes the construction for
$\Omega_k^\text{envi}$ 's and their corresponding $\Omega_k$'s in Algorithm~\ref{alg: algorithm1}.

\subsection{Initialization}
\label{section: initialization}
Due to the nonlinearity of the optimization problem (\ref{variational}), it is possible for ALS to
get stuck at local minima or saddle points. A good initialization is crucial for the success of
ALS. One possibility is to use the ``opening'' procedure in \cite{espig2012note} to obtain each
3-tensors. As mentioned previously, this may suffer ambiguity issue, leading us to
consider a different approach. The proposed initialization procedure consists of two steps. First
we obtain $H^k$'s up to gauges $G^k$'s between them (Algorithm~\ref{alg:initialization}). Then we solve
$d$ least-squares problem to fix the gauges between the $H^k$'s (Algorithm~\ref{alg:gauge}). More
precisely, after Algorithm~\ref{alg:initialization}, we want to use $T^{k,C}$ as $H^k$. However, as in
any factorization, SVD can only determine the factorization of $T^{k,C}$ up to gauge transformations, as
shown in Figure~\ref{figure:fix gauge}. Therefore, between $T^{k,C}$ and $T^{k+1,C}$, some appropriate
gauge $G^k$ has to be inserted (Figure~\ref{figure:fix gauge}).
	
After gauge fixing, we complete the initialization step in Algorithm~\ref{alg: algorithm1}. Before moving
on, we demonstrate the superiority of this initialization v.s. random initialization. In
Figure~\ref{figure:conv} we plot the error between TR and the full function v.s. number of iterations
in ALS, when using the proposed initialization and random initialization. By random initialization,
we mean the $H^k$'s are initialized by sampling their entries independently from the normal
distribution. Then ALS is performed on the example detailed in Section \ref{section:PDE} with
$n=3,d=12$. We set the TR rank to be $r=3$. As we can see, after one iteration of ALS, we already
obtain $10^{-4}$ error using our proposed method, whereas with random initialization, the
convergence of ALS is slower and the solution has a lower accuracy.
	
\begin{figure}[!ht]
  
  \begin{center}
	\includegraphics[width = 0.42\columnwidth,trim = 8cm 6cm 7cm 6cm,clip]{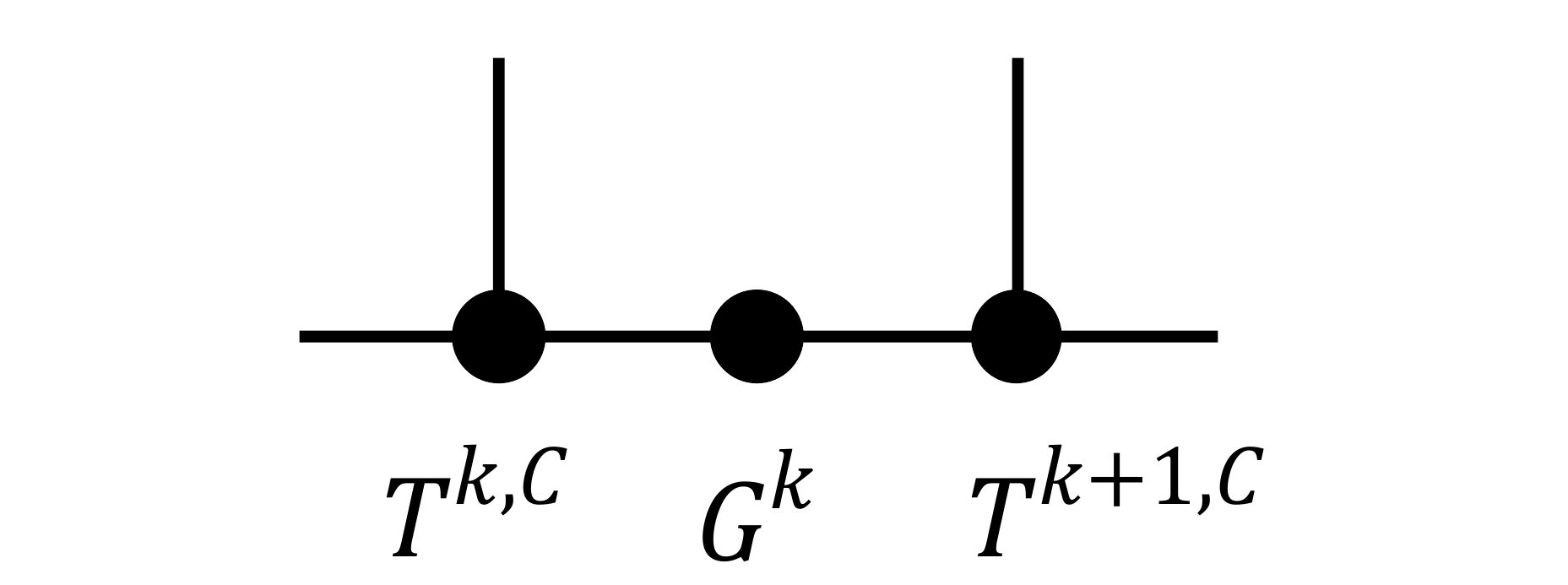}
  \end{center}
  \vspace{-2em}
  \caption{A gauge $G^k$ needs to be inserted between $T^{k,C}$ and $T^{k+1,C}$}\label{figure:fix gauge}
\end{figure}
\begin{figure}[!h]
  \vspace{-2em}
  \begin{center}
	\includegraphics[width=0.5\textwidth]{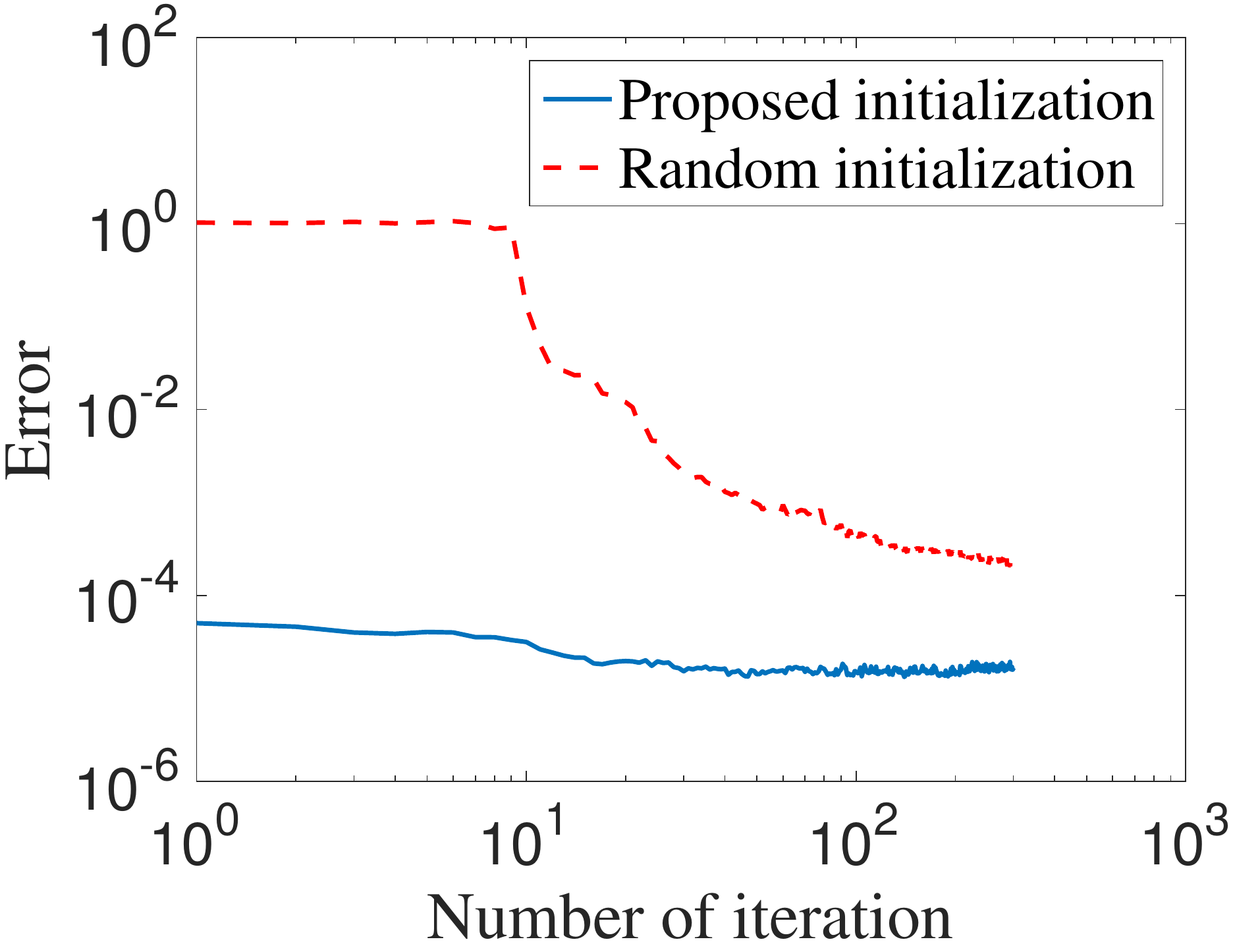}
  \end{center}
  \vspace{-1em}
  \caption{Plot of convergence of the ALS using both random and the proposed initializations for the
    numerical example given in Section \ref{section:PDE} with $n=3,d=12$. The error measure is
    defined in \eqref{error metric}.}\label{figure:conv}
\end{figure}

\begin{algorithm}[H]
  \caption{}\label{alg:initialization}
  \begin{algorithmic}[1]
	\Require 
	\Statex Function $f:[n]^d\rightarrow \mathbb{R}$.
	\Ensure
	\Statex $T^{k,L}\in\mathbb{R}^{n\times r},T^{k,C}\in \mathbb{R}^{r\times n\times r}, T^{k,R}\in\mathbb{R}^{r\times n}$, $k\in[d]$.
	\Statex \textbf{for} {$k=1$ to $k=d$}
	\State \hspace{\algorithmicindent}   \parbox[t]{\dimexpr\linewidth-\algorithmicindent}{Pick an arbitrary $z\in [n]^{d-3}$ and let
	  \begin{equation}
		\Omega_k^\text{ini} := \bigl\{x \in [n]^d \ \vert \ x_{[d]\setminus \{k-1,k,k+1\}} = z, x_{k-1},x_k,x_{k+1} \in [n] \bigr\}.
	  \end{equation}
	  Define
	  \begin{equation}
		T^k := f(\Omega_k^\text{ini} ) \in \mathbb{R}^{n\times n \times n}
	  \end{equation}
	  where the first, second and third dimensions of $T^k$ correspond to the $(k-1),k,(k+1)$-th
      dimensions of $f$. Note that we only pick one $z$ in $\Omega^{\text{envi}}_k$, which is the
      key that we can use SVD procedure in the next step and avoid ambiguity in the
      initialization. The justification of such procedure can be found in
      Appendix~\ref{section:Motivation}.}
    \State \hspace{\algorithmicindent} \parbox[t]{\dimexpr\linewidth-\algorithmicindent}{Now we want
      to factorize the 3-tensor $T^k$ into a tensor train with three nodes using SVD. First treat
      $T^k$ as a matrix by treating the first leg as rows and the second and third legs as
      columns. Apply a rank-$r$ approximation to $T^k$ using SVD:
	  \begin{equation}
		T^k_{1;2,3} \approx U_L \Sigma_L V_L^T.
	  \end{equation}
	  Let $C^{k}\in \mathbb{R}^{r\times n\times n}$ be reshaped from $\Sigma_L V_L^T \in
      \mathbb{R}^{r\times n^2}$.}
    \State \hspace{\algorithmicindent} \parbox[t]{\dimexpr\linewidth-\algorithmicindent}{Treat
      $C^{k}$ as a matrix by treating the first and second legs as rows and third leg as
      columns. Apply SVD to obtain a rank-$r$ approximation:
	  \begin{equation}
		C^k_{1,2;3} \approx U_R \Sigma_R V_R^T.
	  \end{equation}
	  Let $\tilde T^{k,C} \in \mathbb{R}^{r\times n \times r}$ be reshaped from $U_R\Sigma_R\in
      \mathbb{R}^{rn \times r}$.}
    \State \hspace{\algorithmicindent} \parbox[t]{\dimexpr\linewidth-\algorithmicindent}{Let
      $T^{k,L} := U_L \Sigma_L^{1/2} $ and $T^{k,R} := \Sigma_R^{1/2} V_R^T$.  Let $T^{k,C}$ be
      defined by
	  \begin{center}
		\includegraphics[width = 0.8\columnwidth]{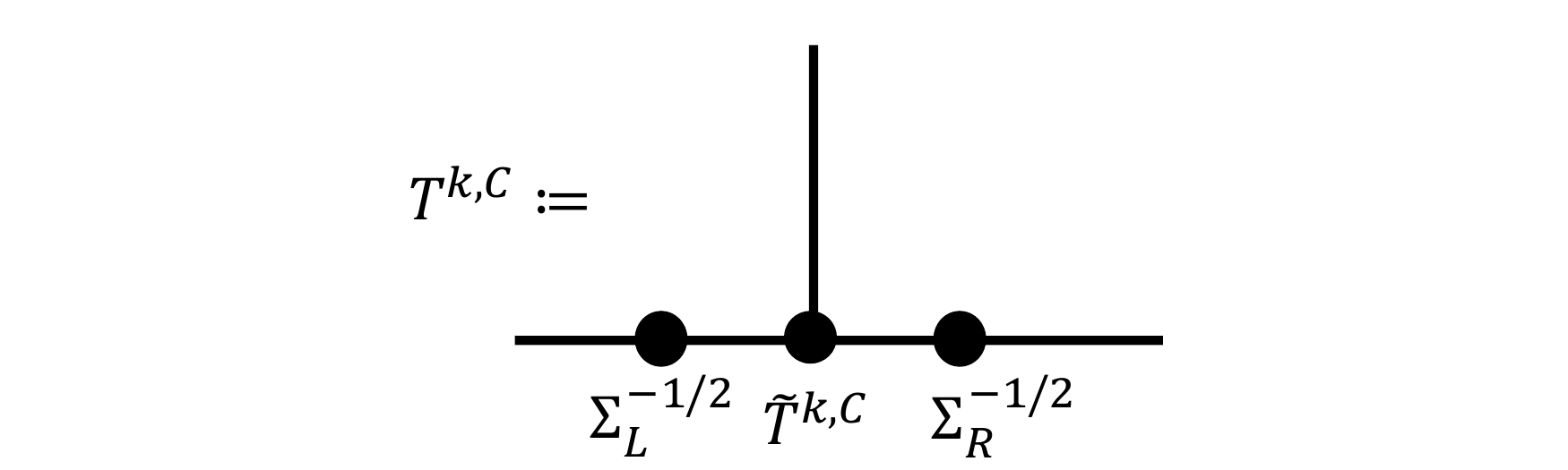}
	  \end{center}
      3-tensor $T^k$ is thus approximated by a tensor train with three tensors
      $T^{k,L}\in\mathbb{R}^{n\times r},T^{k,C}\in \mathbb{R}^{r\times n\times r},
      T^{k,R}\in\mathbb{R}^{r\times n}$.}  \Statex \textbf{end for}
  \end{algorithmic}
\end{algorithm}

\begin{algorithm}[H]
  \caption{}\label{alg:gauge}
  \begin{algorithmic}[1]
	\Require 
	\Statex Function $f:[n]^d\rightarrow \mathbb{R}$, $T^{k,L}, T^{k,C}, T^{k,R}$ for $k\in [d]$ from Algorithm~\ref{alg:initialization}.
	\Ensure
	\Statex Initialization $H^k, k\in[d]$.
	\Statex \textbf{for} {$k=1$ to $k=d$}
	\State
    \hspace{\algorithmicindent} \parbox[t]{\dimexpr\linewidth-\algorithmicindent}{Pick an arbitrary
      $z \in [n]^{d-4}$ and let \begin{equation} \Omega_k^\text{gauge} := \bigl\{x \in [n]^d \ \vert
        \ x_{[d]\setminus \{k-1,k,k+1,k+2\}} = z,\,  \forall\,x_{k-1},x_k,x_{k+1} ,x_{k+2}\in [n]
        \bigr\}
	  \end{equation}
	  and sample
	  \begin{equation}
		S^k = f(\Omega_k^\text{gauge}) \in \mathbb{R}^{n\times n\times n\times n}.
	\end{equation}}
	\State \hspace{\algorithmicindent} \parbox[t]{\dimexpr\linewidth-\algorithmicindent}{Solve
      the least-squares problem
	  \begin{equation}
		G^k = \underset{G}{\text{argmin}}\ \| L^k_{1,2;3} G R^k_{1; 2, 3} - S^k_{1,2;3,4} \|_F^2
	  \end{equation}
	  where $L^k$ and $R^k$ are defined as
	  \begin{center}
		\includegraphics[width = 0.6\columnwidth,trim = 6cm 6cm 4cm 5cm,clip]{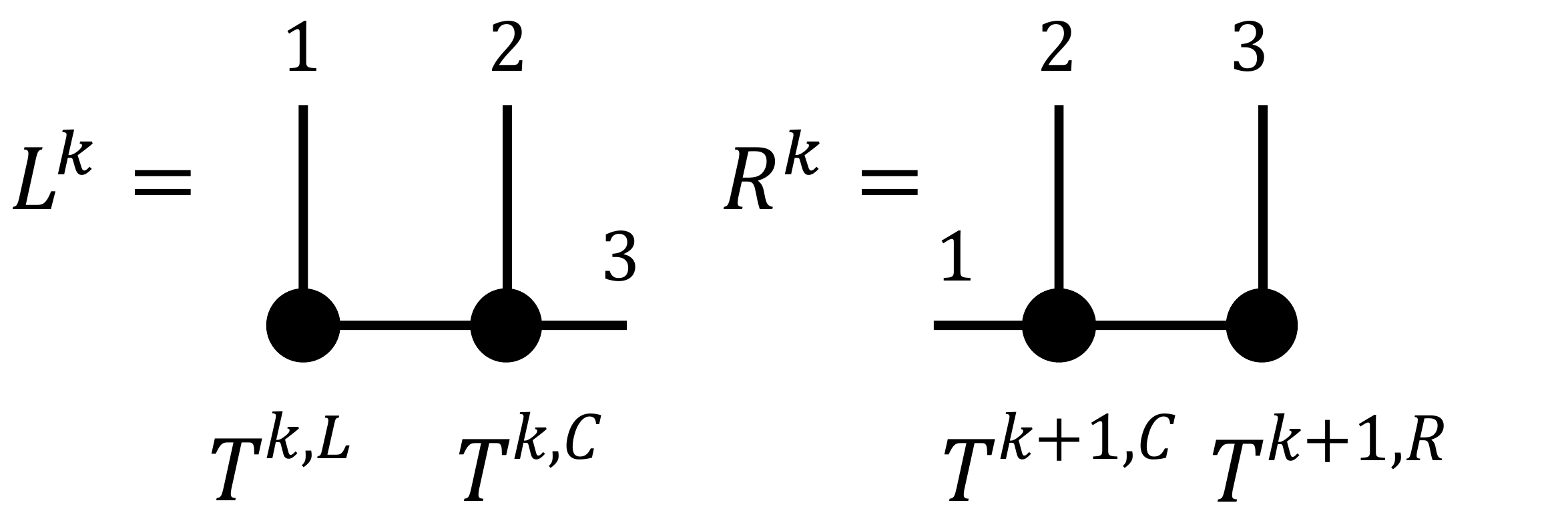}
	  \end{center}
    }
	\State \hspace{\algorithmicindent}   \parbox[t]{\dimexpr\linewidth-\algorithmicindent}{Obtain $H^{k}$:
	  \begin{center}
		\includegraphics[width = 0.4\columnwidth,trim = 9cm 6cm 12cm 7cm,clip]{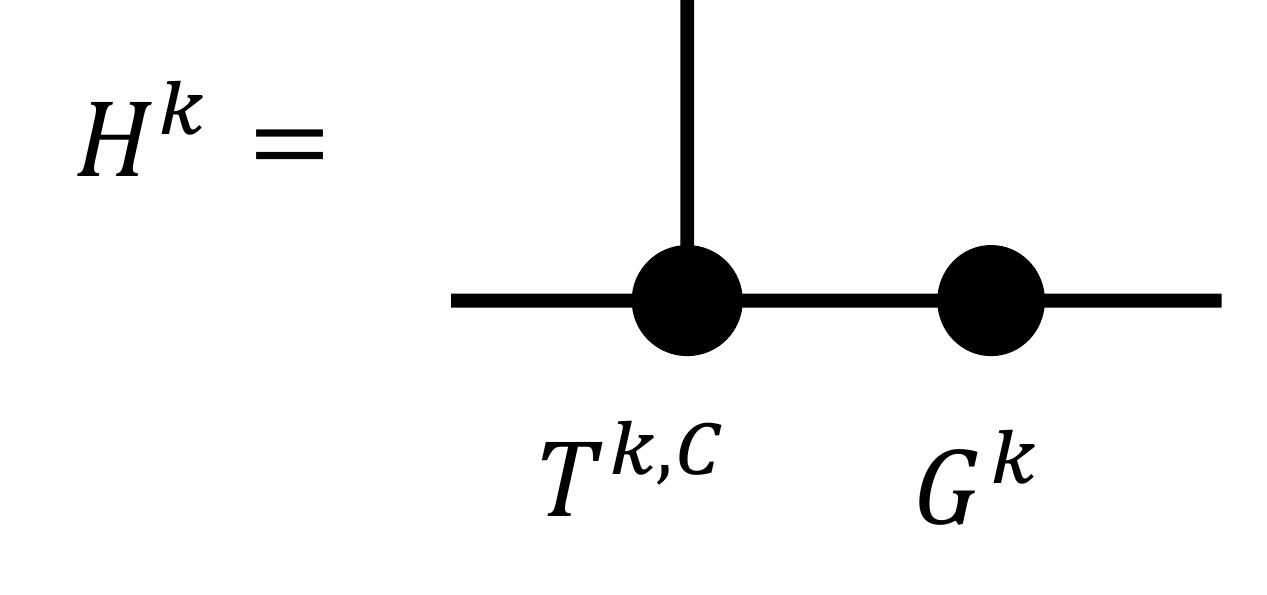}
	  \end{center}
    }
	\Statex \textbf{end for}
  \end{algorithmic}
\end{algorithm}

\subsection{Alternating least-squares}\label{section:ALS}
After constructing $\Omega_k$ and initializing $H^k$, $k\in[d]$, we start ALS by solving problem
(\ref{k-th least squares indep}) at each iteration. This completes Algorithm~\ref{alg: algorithm1}.
	
When running ALS, sometimes we want to increase the TR-rank to obtain a higher accuracy
approximation to the function $f$. In this case, we simply add a row and column of random entries to
each $H^k$, i.e.
\begin{equation}
  H^k(:, i, :) \leftarrow \begin{bmatrix} H^k(:, i, :) & \epsilon_1^{i,k} \\
    \epsilon_2^{i,k} &1 \end{bmatrix},\quad i=1,\ldots,n,\ k=1,\ldots,d,
\end{equation}
where each entry of $\epsilon_1^{i,k}\in \mathbb{R}^{r\times 1},\epsilon_2^{i,k}\in
\mathbb{R}^{1\times r}$ is sampled from Gaussian distribution, and continue with the ALS procedure
with the new $H^k$'s until the error stops decreasing. The variance of each Gaussian random variable
is typically set to $10^{-8}$.


\section{Motivation of the initialization procedure} \label{section:Motivation}
In this section, we motivate our initialization procedure in Alg.  \ref{alg:initialization}. The main idea is by fixing a random index set, a portion of the ring can be singled out and extracted.  To this end, we place the following assumption on the TR $f$.

\begin{assumption}\label{assumption:CI}
	Let the TR $f$ be partitioned into four disjoint regions (Fig \ref{figure:partition}): Regions
	$a$, $b$, $c_1$ and $c_2$ where $a,b,c_1,c_2 \subset [d]$. Regions $a,b,c_1,c_2$ contain $L_a,
	L_b, L_{c_1}, L_{c_2}$ number of dimensions respectively where $L_a+
	L_b+ L_{c_1}+ L_{c_2} = d$.  If $L_{a}, L_{b} \geq L_\text{buffer}$,
	for any $z\in [n]^{L_a+L_b}$, the TR $f$ satisfies
	\begin{equation}
	f(x_{c_1}, x_{a\cup b} ,x_{c_2})\vert_{x_{a\cup b}=z} \propto g(x_{c_1} , x_{a\cup b} )\vert_{x_{a\cup b}=z}  h(x_{a\cup b}, x_{c_2})\vert_{x_{a\cup b}=z} 
	\end{equation}
	for some functions $g,h$. Here ``$\propto$'' denotes the proportional up to a constant relationship.
\end{assumption}
\begin{figure}[!ht]
	\begin{center}
		\includegraphics[width = 0.8\columnwidth,trim = 1cm 1cm 1cm 1cm,clip]{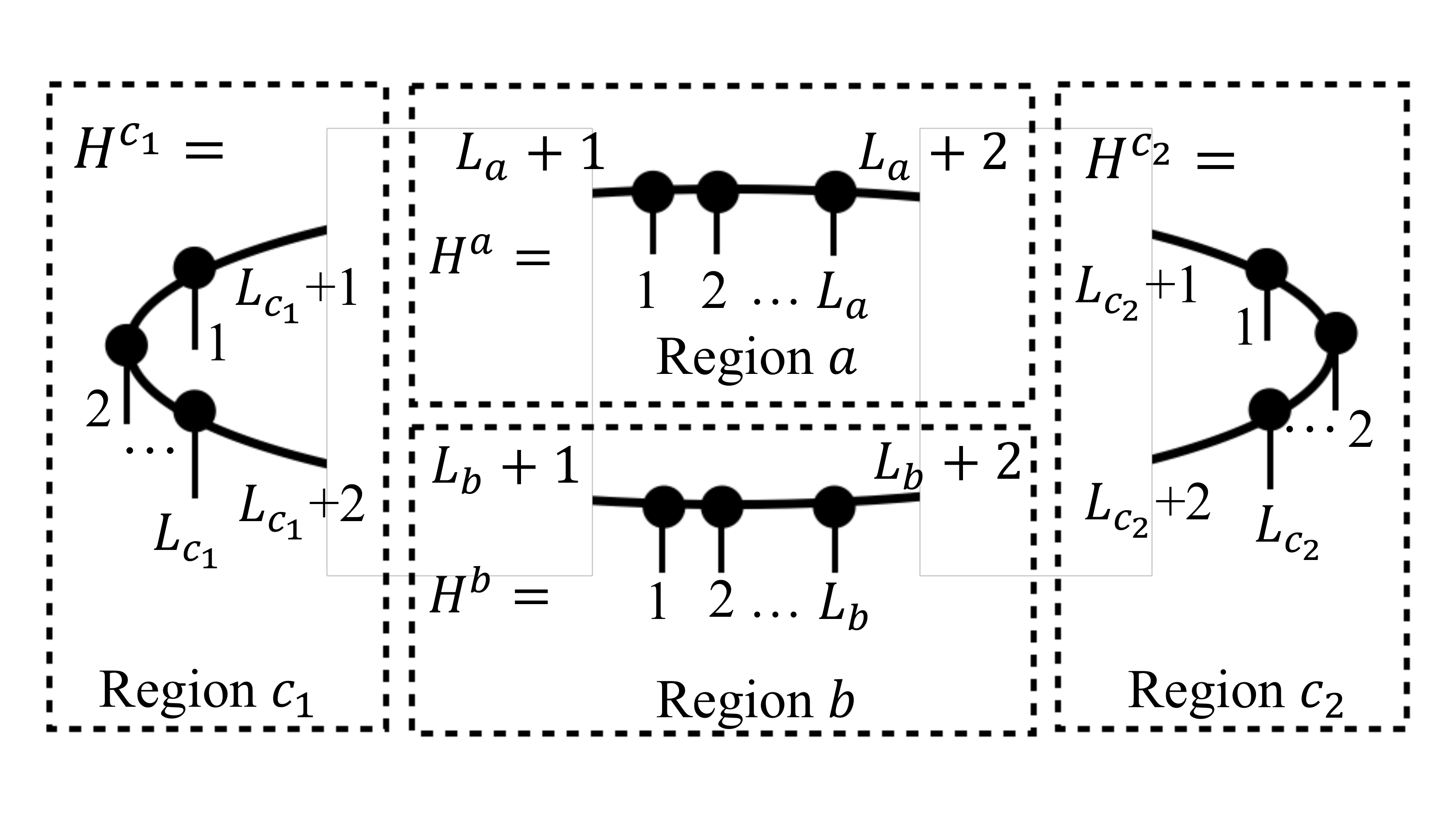}
	\end{center}
	\caption{Figure of TR $f$ partitioned into region $a,b,c_1,c_2$. }\label{figure:partition}
\end{figure}
We note that Assumption \ref{assumption:CI} holds if $f$ is a non-negative function and admits a
Markovian structure. Such functions can arise from Gibbs distribution with energy defined by
short-range interactions \cite{Wolfetal:08}, for example the Ising model.

Next we make certain non-degeneracy assumption on the TR $f$.
\begin{assumption}
	\label{assumption:injectivity}
	Any segment $H$ of the TR $f$ (for example $H^a,H^b, H^{c_1},H^{c_2}$ shown in Figure~\ref{figure:fullrank}), satisfies
	\begin{equation}
	\rank(H_{L+1,L+2; [L]}) = r^2
	\end{equation}
	if $L\geq L_0$ for some natural number $L_0$. In particular, if $L\geq L_0$, we assume the
	condition number of $H_{ [L];L+1,L+2}\geq \kappa$ for some $\kappa = 1+\delta \kappa$, where
	$\delta \kappa\geq 0$ is a small parameter.
	\begin{figure}[!ht]
		\begin{center}
			\includegraphics[width = 0.4\columnwidth,trim = 7cm 6cm 6cm 7cm,clip]{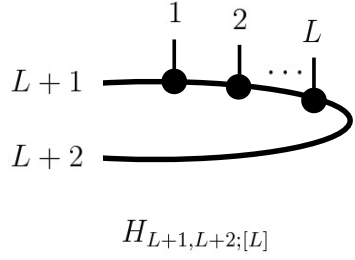}
		\end{center}
		\caption{Figure of a segment of TR, denoted as $H$, with $L+2$ dimensions. The $1,\ldots,L$-th
			dimensions have size $n$, corresponding to outgoing legs of the TR, and the $L+1,L+2$-th
			dimension are the latent dimensions with size $r$. }\label{figure:fullrank}
	\end{figure}
\end{assumption}
Since $H_{L+1,L+2; [L]}\in \mathbb{R}^{r^2 \times n^L}$, it is natural to expect when $n^L \geq
r^2$, $H_{L+1,L+2; [L]}$ is rank $r^2$ generically \cite{perez2006matrix}.

We now state a proposition that leads us to the intuition behind designing the initialization
procedure Algorithm~\ref{alg:initialization}.
\begin{proposition}\label{proposition:disentangle}
	Let
	\begin{equation}
	s^1 = e_{i_1}\otimes e_{i_2} \otimes \cdots \otimes e_{i_{L_{a}}}, \quad s^2 = e_{j_1}\otimes e_{j_2} \otimes \cdots \otimes e_{j_{L_{b}}}
	\end{equation}
	be any two arbitrary sampling vectors where $\{e_k\}_{k=1}^n$ is the canonical basis in
	$\mathbb{R}^n$. If $L_a,L_b,L_{c_1},L_{c_2}\geq \max(L_0,L_\text{buffer})$, the two matrices $B^1,
	B^2\in \mathbb{R}^{r\times r}$ defined in Figure~\ref{figure:B12} are rank-1.
	\begin{figure}[!ht]
		\begin{center}
			\includegraphics[width = 0.8\columnwidth,trim = 0cm 1cm 1cm 1cm,clip]{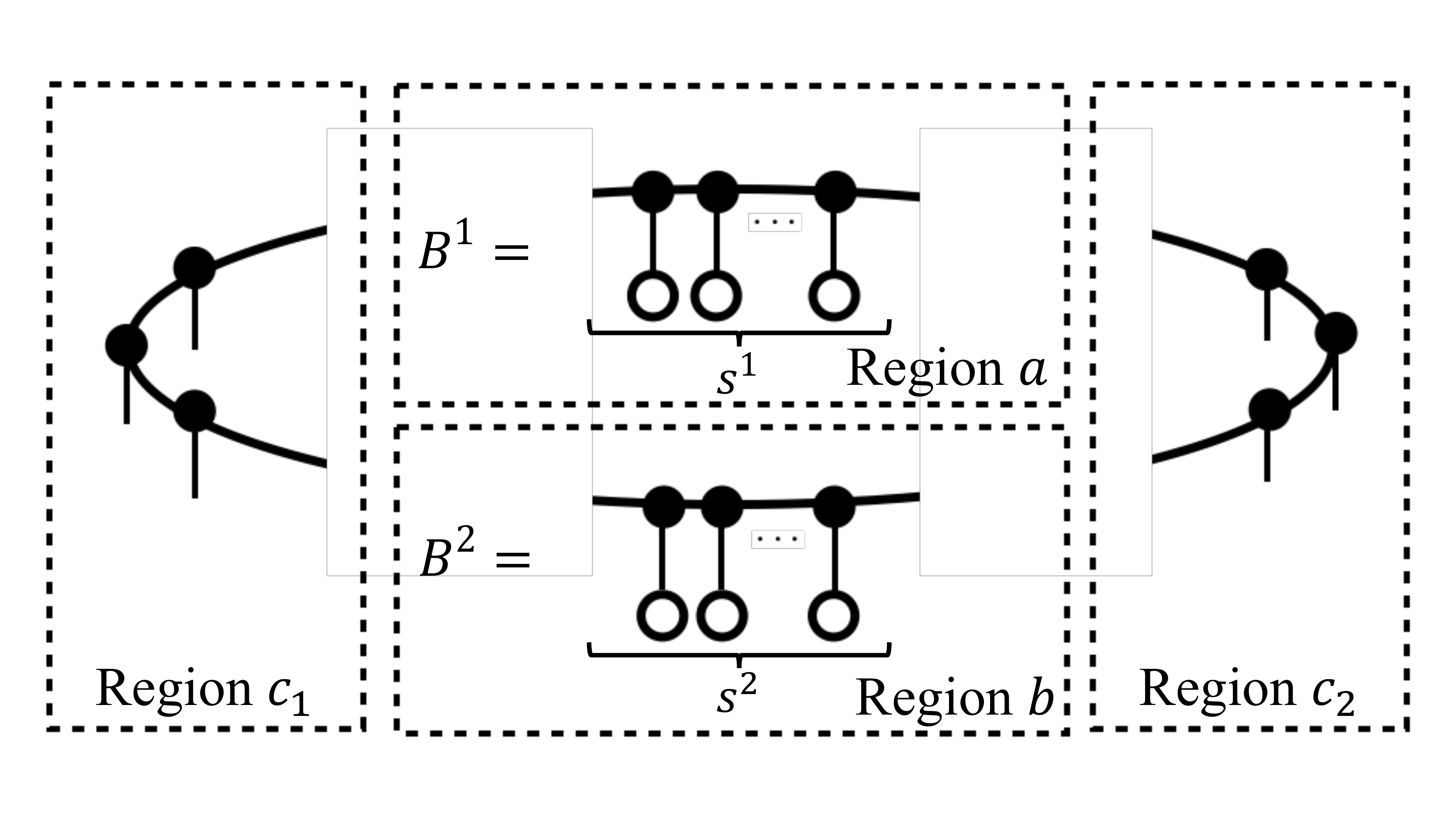}
		\end{center}
		\caption{Definition of the matrices $B^1, B^2$ in Proposition \ref{proposition:disentangle}. }\label{figure:B12}
	\end{figure}
\end{proposition}
\begin{proof}
	Due to Assumption \ref{assumption:injectivity}, $H^{c_1}_{L_{c_1}+1,L_{c_1}+2;[L_{c_1}]} \in
	\mathbb{R}^{r^2\times n^{L_{c_1}}}$ and $H^{c_2}_{L_{c_2}+1,L_{c_2}+2;[L_{c_2}]}$ $\in
	\mathbb{R}^{r^2\times n^{L_{c_2}}}$ defined in Figure~\ref{figure:B12} are rank-$r^2$. Along with
	the implication of Assumption \ref{assumption:CI} that
	\begin{equation}
	\rank\bigl(\bigl(H^{c_1}_{L_{c_1}+1,L_{c_1}+2;[L_{c_1}]}\bigr)^T B^1 \otimes B^2{H^{c_2}_{L_{c_2}+1,L_{c_2}+2;[L_{c_2}]}}\bigr) = 1,
	\end{equation}
	we get 
	\begin{equation}
	\label{kron rank 1}
	\rank(B^1 \otimes B^2)=1.
	\end{equation}
	Since $\rank(B^1)\rank(B^2) = \rank(B^1\otimes B^2)=1$, it follows that the rank of $B^1$, $B^2$ are 1. 
\end{proof}




The conclusion of Proposition \ref{proposition:disentangle} implies that to obtain the segment of TR
in region $a$, one simply needs to apply some sampling vector $s^2$ in the canonical basis to region
$b$ to obtain the configuration in Figure~\ref{figure:openring} where the vectors $p^b,q^b \in
\mathbb{R}^{r}$. Our goal is to extract the nodes in region $a$ as $H^k$'s. It is intuitively
obvious that one can apply the TT-SVD technique in \cite{oseledets2010tt} to extract them. Such
technique is indeed used in the proposed initialization procedure where we assume $L_\text{buffer} =
1, L_0 = 1, L_a = 1, L_b = d-3$. For completeness, in Proposition \ref{proposition:stability} in the appendix, we
formalize the fact that one can use TT-SVD to learn each individual 3-tensor in the TR $f$ up to
some gauges. We further provide a perturbation analysis for the case when Markovian-type assumption
holds only approximately in Proposition \ref{proposition:stability}.

\begin{figure}[!ht]
	\begin{center}
		\includegraphics[width = 0.6\columnwidth,trim = 0cm 2cm 0cm 3cm,clip]{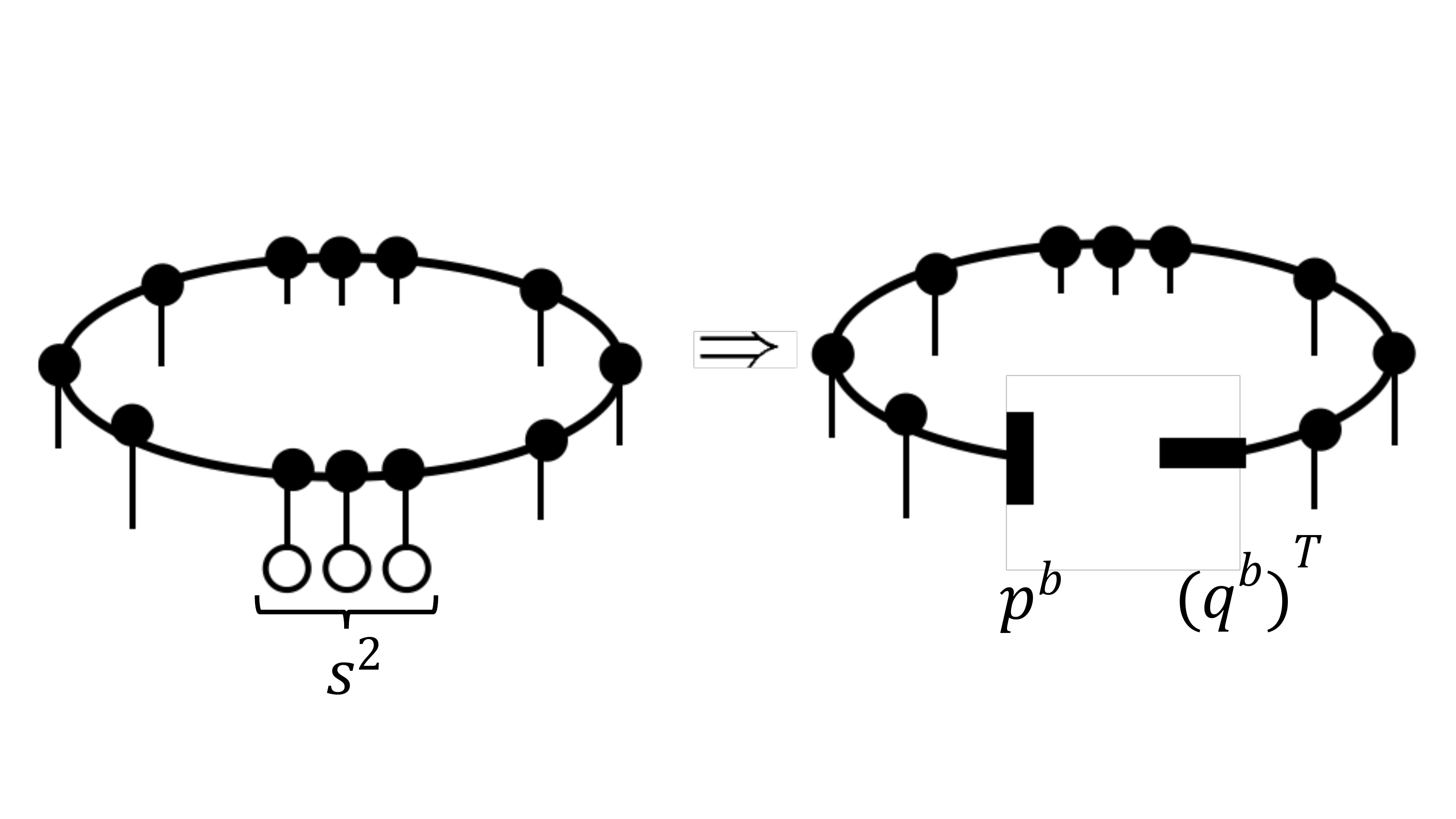}
	\end{center}
	\caption{Applying a sampling vector $s^2$ in the canonical basis to region $b$ gives the TT. }\label{figure:openring}
\end{figure}

\section{Numerical results}\label{section:numerical}
In this section, we present numerical results on the proposed method for tensor ring
decomposition. We calculate the error between the obtained tensor ring
decomposition and function $f$ as:
\begin{equation}
  \label{error metric}
  E = \sqrt{\frac{\sum_{x\in \Omega}  \big( \Tr(H^1[x_1]\cdots H^d[x_d])- f(x_1,\ldots,x_d)\big)^2}{\sum_{x\in \Omega} f(x_1,\ldots,x_d)^2}}.
\end{equation}  
Whenever it is feasible, we let $\Omega = [n]^d$. Otherwise, we subsample $\Omega$ from $[n]^d$ at random: For every $x\in\Omega$, $x_i$ is drawn from $[n]$ uniformly at random. If the dimensionality of $f$ is large, we simply sample $\Omega$ from $[n]^d$ at random. For the proposed algorithm, we also measure the error on the entries sampled for learning TR as:
\begin{equation}
  E_\text{skeleton} = \sqrt{\frac{\sum_{x\in \cup_k \Omega_k}  \big( \Tr(H^1[x_1]\cdots H^d[x_d])- f(x_1,\ldots,x_d)\big)^2}{\sum_{x\in\cup_k \Omega_k} f(x_1,\ldots,x_d)^2}}.
\end{equation} 

In the experiments, we compare our method, denoted as TR-ALS+, with TR-ALS proposed in \cite{wang2017efficient}. In  \cite{wang2017efficient}, the cost in \eqref{full variational} is minimized using ALS where \eqref{k-th least squares} is solved for each $k$ in an alternating fashion. Although \cite{wang2017efficient} proposed an SVD based initialization approach similar to the recursive SVD algorithm for TT \cite{oseledets2010tt}, this method has exponential complexity in $d$. Therefore the comparison with such an initialization is omitted and we use randomized intialization for TR-ALS. As we shall see, TR-ALS+ is generally an order of magnitude faster than TR-ALS, due to the special structure of the samples. For each experiment we run both TR-ALS and TR-ALS+  for five times and report the median accuracy. For TR-ALS, we often have to use less samples such that the running time is not excessively long (recall that TR-ALS has $O(d^2)$ complexity per iteration). We also compare ourselves with the DMRG-Cross algorithm \cite{savostyanov2011fast} (which gives a TT). As a method that is based on interpolative decomposition, DMRG-Cross is able to obtain high quality approximation if we allow a large TT-rank representation. Since we obtain the TR based on ALS optimization, the accuracy may not be comparable to DMRG-Cross. What we want to emphasize here is that if the given situation only requires moderate accuracy, our method could give a more economical representation than TT obtained from DMRG-Cross. To convey this message, we set the accuracy of DMRG-Cross so that it matches the accuracy of our proposed TR-ALS. 

\subsection{Example 1: A toy example}
We first compress the function 
\begin{equation}
  f(x_1,\ldots,x_d) = \frac{1}{\sqrt{1+x_1^2+\ldots+x_d^2}},\quad x_k\in [0,1]
\end{equation}
considered in \cite{espig2012note} into a tensor ring.  In this example, we let $s=4$ (recall that $s$ is the size of $\Omega_k^\text{envi}$) in TR-ALS+. The number of samples we can afford to use for TR-ALS is less than TR-ALS+ due to the excessively long running time since each iteration of TR-ALS has a complexity scaling of $O(d^2)$. In this example, although sometimes TR-ALS+ has lower accuracy than TR-ALS, the running time of TR-ALS+ is significantly shorter. In particular, for the case when $d=12$, TR-ALS fails to converge using the same amount of samples as TR-ALS+. Both TR-ALS+ and TR-ALS give TR with tensor components with smaller sizes than TT. The error $E$ reported for the case of $d=12$ is obtained from sampling $10^5$ entires of the tensor $f$. 

\begin{table}[H]
  \centering 
  \scalebox{0.9}{
	\begin{tabular}{c c c c c c c} 
	  \hline\hline 
	  Setting & Format & {\begin{tabular}[c]{@{}l@{}}Rank \\ $(r_1,\ldots,r_d)$\end{tabular}}  & $E_\text{skeleton}$ & $E$ & $ \frac{\text{Number of observations}}{n^d}$ & Run Time (s)\\ 
	  \hline\hline 
	  $d=6, n=10$ & TR-ALS+  & (3,3,3,3,3,3)  & 2.3e-03 & 6.3e-04  & 1.8e-01 & 4.7\\
	  & TR-ALS & (3,3,3,3,3,3)  & 4.3e-05 & 4.5e-05 &2.8e-02 & 1360\\
	  & TT  &  (5,5,5,5,5,1) &    -     &    1.2e-04  & - &  2.4 \\
	  \hline
	  $d=6, n=20$ & TR-ALS+  & (3,3,3,3,3,3)  & 5.1e-04 & 9.4e-05 & 2.1e-02 & 24 \\
	  & TR-ALS & (3,3,3,3,3,3)  & 5.0e-05 & 5.4e-05  & 8.2e-04 & 2757\\
	  & TT  & (5,5,6,5,5,1)  &  -       & 6.8e-05    & - & 7.1  \\
	  \hline
	  $d=12, n=5$ & TR-ALS+  & \begin{tabular}{@{}c@{}}(3,3,3,3,3,3 \\ \ \ 3,3,3,3,3,3)\end{tabular}  &  7.1e-04 & 5.9e-04 & 1.7e-04  & 28\\
				& TT-ALS  &  \begin{tabular}{@{}c@{}}(3,3,3,3,3,3 \\ \ \ 3,3,3,3,3,3)\end{tabular} &   0.97   & 0.97  &1.7e-04 & 3132  \\
				& TT  &  \begin{tabular}{@{}c@{}}(5,6,6,6,6,6 \\ \ \ 6,6,5,5,5,1) \end{tabular}   &     -    & 2.2e-05    & - & 2.9  \\
	  \hline	
  \end{tabular}} \\ [1ex]
  \caption{Results for Example 1. $n$ corresponds to the number of uniform grid points on $[0,1]$ for each
  	$x_k$. The tuple $(r_1,\ldots,r_d)$ indicates the rank of the learnt
    TR and TT. $E_\text{skeleton}$  is computed on the samples used for learning the TR. }\label{table: results1}
\end{table}

\subsection{Example 2: Ising spin glass}
In this example, we demosntrate the advantage of TR-ALS+ in compressing high-dimensional function arising from many-body physics, the traditional field where TT or MPS is extensively used \cite{AKLT:88,white1992density}. We consider compressing the free energy of Ising spin glass with a ring geometry:
\begin{equation}
  f(J_1,\ldots,J_d) = -\frac{1}{\beta}\log\bigg[\Tr\bigg(\prod_{i=1}^d \begin{bmatrix} e^{\beta J_i} & e^{-\beta J_i}\\ e^{-\beta J_i} & e^{\beta J_i}\end{bmatrix}  \bigg)\bigg].
\end{equation}
We let $\beta = 10$, and $J_i \in \{-2.5,-1.5,1,2\}, i\in[d]$. This corresponds to Ising model with
temperature of about 0.1K. We let the number of environment samples $s=5$. When computing the error $E$ for the case of $d=24$, due to the size of $f$, we simply sub-sample $10^5$ entries of $f$ where $J_i$'s are sampled independently and uniformly from $\{-2.5,-1.5,1,2\}$. For $d=12$, the solution obtained by TR-ALS+ is superior due to the initialization procedure. We see that in both $d=12,24$ cases, the running time of TR-ALS is much longer compare to TR-ALS+.
	
\begin{table}[H]
  \centering 
  \scalebox{0.9}{
	\begin{tabular}{c c c c c c c} 
	  \hline\hline 
	  Setting & Format & {\begin{tabular}[c]{@{}l@{}}Rank \\ $(r_1,\ldots,r_d)$\end{tabular}}  & $E_\text{skeleton}$ &$E$& $ \frac{\text{Number of observations}}{n^d}$& Run Time (s)\\ 
	  \hline\hline 
	  $d=12, n=4$ & TR-ALS+  & \begin{tabular}{@{}c@{}}(4,4,4,4,4,4 \\ \ \ 4,4,4,4,4,4)\end{tabular}  &  3.9e-03 & 3.8e-03 & 1.6e-02 & 7 \\
	  & TR-ALS  & \begin{tabular}{@{}c@{}}(4,4,4,4,4,4 \\ \ \ 4,4,4,4,4,4)\end{tabular}  &  4.4e-02 & 5.2e-02 & 1.6e-02 & 994 \\
	  & TT  &  \begin{tabular}{@{}c@{}}(6,7,7,7,7,7 \\ \ \ 7,7,7,6,4,1)\end{tabular} &    -     &  4.2e-03  & - &  2.8  \\
	  \hline	
	  $d=24, n=4$ & TR-ALS+  & \begin{tabular}{@{}c@{}}(3,3,3,3,3,3 \\ \ 3,3,3,3,3,3 \\ \  3,3,3,3,3,3 \\ \ \ 3,3,3,3,3,3) \end{tabular}  & 4.8e-03 & 2.7e-03 & 1.6e-10& 19 \\
	& TR -ALS & -  &  - & - & 1.6e-10& - \\
	  & TT  &  \begin{tabular}{@{}c@{}}(6,8,8,8,6,6 \\ \ 6,6,6,6,7,6 \\ \ 5,6,6,6,6,7 \\ \ \ 7,6,6,6,4,1) \end{tabular} &    -     &  3.7e-03  &  - & 9.3  \\
	  \hline
		\end{tabular}} \\ [1ex]
  \caption{Results for Example 2. Learning the free energy of Ising spin glass. }\label{table: results2}
\end{table}

\subsection{Example 3: Parametric elliptic partial differential equation (PDE)}
\label{section:PDE}
In this section, we demonstrate the performance of our method in solving parametric PDE. We are
interested in solving elliptic equation with random coefficients
\begin{equation}
  \frac{\partial }{\partial x} a(x) \bigg( \frac{\partial }{\partial x} u(x) + 1\bigg) = 0, \quad x\in [0,1]
\end{equation}
subject to periodic boundary condition, where $a(\cdot)$ is a random field. In particular, we
want to parameterize the effective conductance function
\begin{equation}
  A_\text{eff}(a(\cdot)) := \int_{[0,1]} a(x) \bigg( \frac{\partial }{\partial x} u(x) + 1 \bigg)^2 dx
\end{equation}
as a TR. By discretizing the domain into $d$ segments and assuming $a(x) = \sum_{i=1}^d a_i
\chi_i(x)$, where each $a_i \in [1,2,3]$ and $\chi_i$'s being step functions on uniform intervals on
$[0,1]$, we determine $A_\text{eff}(a_1,\ldots,a_d)$ as a TR. In this case, the effective
coefficients have an analytic solution
\begin{equation}
  A_\text{eff}(a_1,\ldots,a_d) = \bigg(\frac{1}{d}\sum_{i=1}^d a_i \bigg)^{-1}
\end{equation}
and we use this formula to generate samples to learn the TR. For this example, we pick $s=4$. The
results are reported in Table \ref{table: results3}. When computing $E$ with $d=24$,
again $10^5$ entries of $f$ are subsampled where $a_i$'s are sampled independently and uniformly
from $\{1,2,3\}$. We note that although in this situation, there is an analytic formula for the
function we want to learn as a TR, we foresee further usages of our method when solving parametric
PDE with periodic boundary condition, where there is no analytic formula for the physical quantity
of interest (for example for the cases considered in \cite{khoo2017solving}).

\begin{table}[H]
  \centering 
  \scalebox{0.9}{
	\begin{tabular}{c c c c c c c} 
	  \hline\hline 
	  Setting & Format & {\begin{tabular}[c]{@{}l@{}}Rank \\ $(r_1,\ldots,r_d)$\end{tabular}}  & $E_\text{skeleton}$ & $E$ &$ \frac{\text{Number of observations}}{n^d}$ & Run Time (s)\\ 
	  \hline\hline 
	  $d=12, n=3$ & TR-ALS+ & \begin{tabular}{@{}c@{}}(3,3,3,3,3,3 \\ \ \ 3,3,3,3,3,3)\end{tabular}  &  1.1e-05 & 1.1e-05 &1.4e-02 & 22 \\
	  & TR-ALS  & \begin{tabular}{@{}c@{}}(3,3,3,3,3,3 \\ \ \ 3,3,3,3,3,3)\end{tabular}  &  5.7e-06 & 6.8e-06 & 1.4e-02 & 1414 \\
	  & TT  &  \begin{tabular}{@{}c@{}}(5,5,5,5,5,5 \\ \ \ 5,5,5,3,3,1)\end{tabular} &    -     &  2.5e-05  & - & 0.76  \\
	  \hline	
	  $d=24, n=3$ & TR-ALS+  & \begin{tabular}{@{}c@{}}(3,3,3,3,3,3 \\ \ 3,3,3,3,3,3 \\ \  3,3,3,3,3,3 \\ \ \ 3,3,3,3,3,3) \end{tabular}  &  2.6e-05 & 2.8e-05 &5.5e-06 &47 \\
	  & TR-ALS  & - &  - & -  & 5.5e-06  & - \\
	  & TT  &  \begin{tabular}{@{}c@{}}(5,5,5,5,5,5 \\ \ 5,5,5,5,5,5 \\ \ 5,5,5,5,5,5 \\ \ \ 5,5,5,3,3,1) \end{tabular} &    -     &  1.7e-05  &  - &1.5  \\
				\hline
  \end{tabular}} \\ [1ex]
  \caption{Results for Example 3. Solving parametric elliptic PDE.}\label{table: results3}
\end{table}

\section{Conclusion}\label{section:conclusion}
In this paper, we propose method for learning a TR representation based on ALS. Since the problem of
determining a TR is a non-convex optimization problem, we propose an initialization strategy that
helps the convergence of ALS. Furthermore, since using the entire tensor $f$ in the ALS is
infeasible, we propose an efficient hierarchical sampling method to identify the important
samples. Our method provides a more economical representation of the tensor $f$ than TT-format.  As
for future works, we plan to investigate the performance of the algorithms for quantum systems. One
difficulty is that the Assumption~\ref{assumption:CI} (Appendix~\ref{section:Motivation}) for the proposed initialization procedure does not in general hold for quantum systems
with short-range interactions. Instead, a natural assumption for a quantum state exhibiting a
tensor-ring format representation is the exponential correlation decay \cite{HastingsKoma:06,
  BrandaoHorodecki}. The design of efficient algorithms to determine the TR representation under such assumption is left for
future works. Another natural direction is to extend the proposed method to tensor networks in higher spatial dimension,
which we shall also explore in the future.
	
\appendix

\section{Stability of initialization}\label{section:stability}
In this section, we analyze the stability of the proposed
initialization procedure, where we relax Assumption
\ref{assumption:CI} to approximate Markovianity.
\begin{assumption}\label{assumption:CI2}
  Let
  \begin{equation}
	\Omega_z := \bigl\{(x_{c_1},x_{a\cup b}, x_{c_2})\mid x_{c_1} \in [n]^{L_{c_1}}, x_{c_2} \in [n]^{L_{c_2}}, x_{a\cup b} = z \bigr\}
  \end{equation}
  for some given $z\in [n]^{L_a+L_b}$. For any $z\in [n]^{L_a+L_b}$, we assume
  \begin{equation}
	\label{near rank1}
	\frac{\| f(\Omega_z)_{c_1;a\cup b\cup c_2} \|_2^2}{\| f(\Omega_z)_{c_1;a\cup b\cup c_2}\|_F^2} \geq \alpha.
  \end{equation} 
  for some $0< \alpha \leq 1$ if $L_{a},L_{b}\geq L_\text{buffer}$.
\end{assumption}
This assumption is a relaxation of Assumption 1. Indeed, if \eqref{near rank1} holds for $\alpha =
1$, it implies that $f(\Omega_z)_{c_1; a\cup b \cup c_2}$ is rank $1$. Under the
Assumption~\ref{assumption:CI2}, we want to show that using Algorithm~\ref{alg:initialization}, one can
extract $H^k$'s approximately. The final result is stated in Proposition
\ref{proposition:stability}, obtained via the next few lemmas. In particular, we show that when the condition number $\kappa$ of the tensor ring components (defined in Lemma \ref{lemma:rank-1 con}) satisfies $\kappa=1$, as $\alpha\rightarrow 1$, the approximation error goes to $0$. In the first lemma, we show that
$B^1,B^2$ defined in Figure~\ref{figure:B12} are approximately rank-1.

\begin{lemma}\label{lemma:rank-1 con}
  Let $H^{c_1}, H^{c_2}, B^1, B^2$ be defined according to Figure~\ref{figure:partition} and
  \ref{figure:B12}, where the sampling vectors $s^1, s^2$ are defined in Proposition
  \ref{proposition:disentangle}.  If $L_{c_1},L_{c_2},L_a,L_b \geq \max(L_0,L_\text{buffer})$, then
  \begin{equation}
	\label{approx rank-1}
		\frac{\| B^1\|_2^2} {\|B^1\|_F^2}, \frac{\| B^2\|_2^2} {\|B^2\|_F^2}\geq
        \frac{\alpha}{\kappa^4}.
  \end{equation}
\end{lemma}
\begin{proof}
  By Assumption~\ref{assumption:CI2}, 
  \begin{align}
	\alpha&\leq\frac{\bigl\lVert \bigl(H^{c_1}_{L_{c_1}+1,L_{c_1}+2;[L_{c_1}]}\bigr)^T B^1 \otimes B^2{H^{c_2}_{L_{c_2}+1,L_{c_2}+2;[L_{c_2}]}} \bigr\rVert_2^2}{\bigl\lVert \bigl(H^{c_1}_{L_{c_1}+1,L_{c_1}+2;[L_{c_1}]}\bigr)^T B^1 \otimes B^2{H^{c_2}_{L_{c_2}+1,L_{c_2}+2;[L_{c_2}]}} \bigr\rVert_F^2}\cr
	&\leq \kappa_{c_1}^2 \kappa_{c_2}^2 \frac{\| B^1\otimes B^2\|_2^2 }{\| B^1\otimes B^2\|_F^2} \cr
	&=\kappa_{c_1}^2 \kappa_{c_2}^2 \frac{\| B^1\|_2^2} {\|B^1\|_F^2}\frac{\| B^2\|_2^2} {\|B^2\|_F^2},
  \end{align}
  where $\kappa_{c_1}, \kappa_{c_2}\leq \kappa$ are condition numbers of
  $H^{c_1}_{L_{c_1}+1,L_{c_1}+2;[L_{c_1}]}$ and $H^{c_2}_{L_{c_2}+1,L_{c_2}+2;[L_{c_2}]}$
  respectively.
\end{proof}
Let $p^b({q^b})^T$ be the best rank-1 approximation to $B^2$. Before registering the next corollary,
we define $H^{[d]\setminus b}$ and $\tilde H^{[d]\setminus a}$ in Figure~\ref{figure:Hdef}.
\begin{figure}[!ht]
  \begin{center}
	\includegraphics[width = 0.7\columnwidth,trim = 0cm 0cm 0cm 0cm,clip]{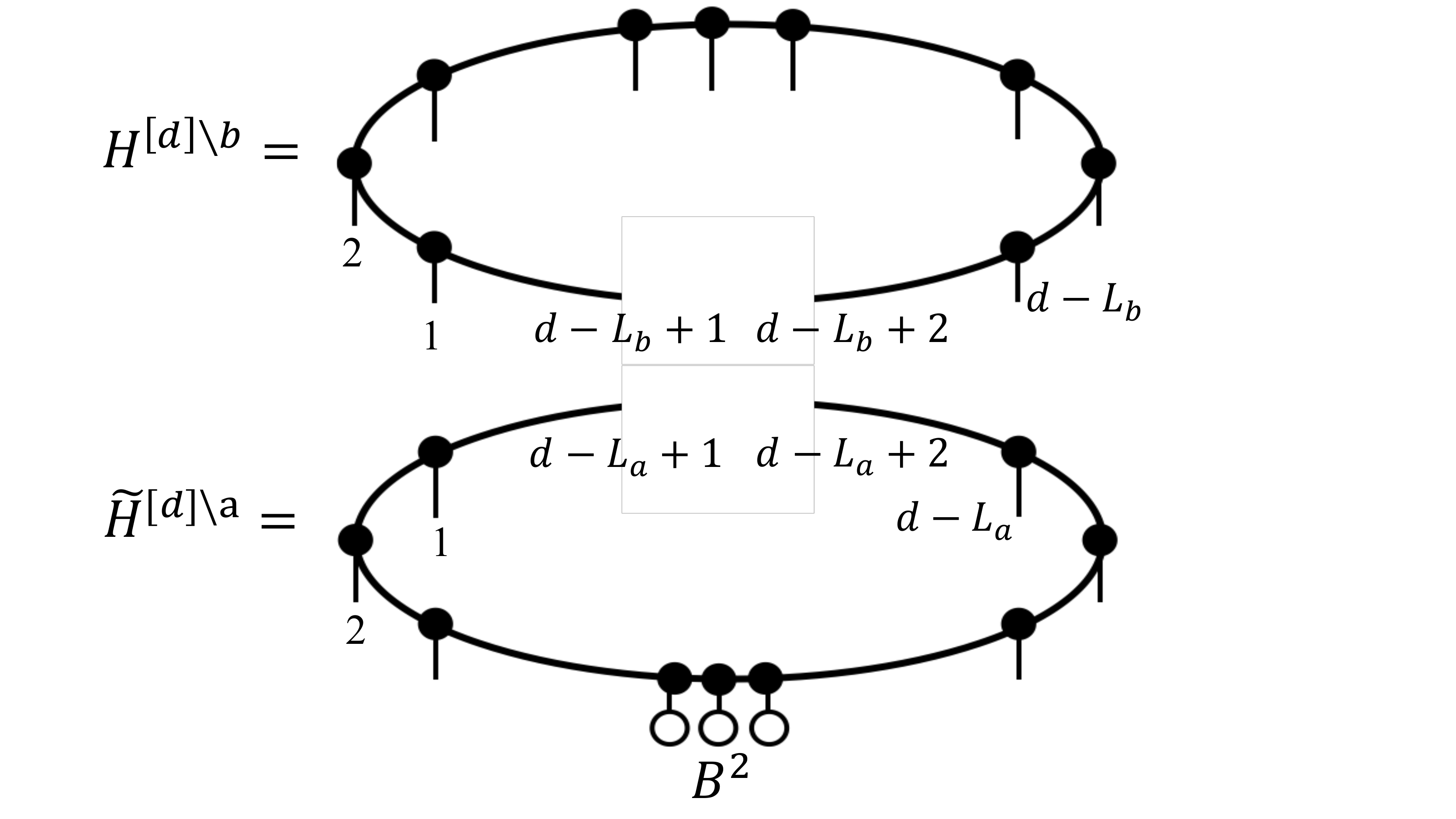}
  \end{center}
  \caption{Definition of $H^{[d]\setminus b}$ and $\tilde H^{[d]\setminus a}$. }\label{figure:Hdef}
\end{figure}
\begin{corollary}\label{corollary:two TT}
  Under the assumptions of Lemma \ref{lemma:rank-1 con}, for any sampling operator $s^2$ defined in
  Proposition \ref{proposition:disentangle},
  \begin{equation}
	\frac{\|H^{[d]\setminus b}_{[d-L_b];d-L_b+1,d-L_b+2} \text{vec}(p^b({q^b})^T) - f_{[d]\setminus b;b}s^2\|_2^2}{\|f_{[d]\setminus b;b}s^2\|_F^2}\leq \kappa^2(1 - \frac{\alpha}{\kappa^4}).
  \end{equation}
\end{corollary}
\begin{proof}
  Lemma \ref{lemma:rank-1 con} implies 
  \begin{align}
	& \frac{\| H^{b}_{L_{b}+1,L_{b}+2;[L_{b}]} s^2 - \text{vec}(p^b ({q^b})^T) \|_2^2}{\| H^{b}_{L_{b}+1,L_{b}+2;[L_{b}]} s^2\|_2^2} \cr
	&\qquad = \frac{\|B^2 - p^b ({q^b})^T\|_F^2}{\|B^2\|_F^2} = \frac{\|B^2\|_F^2 -\| p^b ({q^b})^T\|_F^2}{\|B^2\|_F^2} \leq 1-\frac{\alpha}{\kappa^4}.
  \end{align}
  Then
  \begin{align}
	& \frac{\|H^{[d]\setminus b}_{[d-L_b];d-L_b+1,d-L_b+2} \text{vec}(p^b({q^b})^T) - f_{[d]\setminus b;b}s^2\|_2^2}{\|f_{[d]\setminus b;b}s^2\|_2^2} \cr
	&\qquad \leq  \frac{\|H^{[d]\setminus b}_{[d-L_b];d-L_b+1,d-L_b+2}\|_2^2  \| H^{b}_{L_{b}+1,L_{b}+2;[L_{b}]} s^2 - \text{vec}(p^b ({q^b})^T) \|_2^2}{\| H^{[d]\setminus b}_{[d-L_b];d-L_b+1,d-L_b+2} H^{b}_{L_{b}+1,L_{b}+2;[L_{b}]} s^2 \|_2^2}\cr
	&\qquad \leq  \kappa_{[d]\setminus b}^2 \frac{\| H^{b}_{L_{b}+1,L_{b}+2;[L_{b}]} s^2 - \text{vec}(p^b ({q^b})^T) \|_2^2}{\| H^{b}_{L_{b}+1,L_{b}+2;[L_{b}]} s^2\|_2^2}
  \end{align}
  where $\kappa_{[d]\setminus b}^2$ is the condition number of $H^{[d]\setminus b}_{[d-L_b];d-L_b+1,d-L_b+2}$. Recall that $H^b$ is defined in Figure~\ref{figure:partition}.
\end{proof}
This corollary states that the situation in Figure~\ref{figure:openring} holds approximately. More
precisely, let $T, \hat T\in \mathbb{R}^{n^{d-L_b}}$ be defined as
\begin{equation}\label{Tdef}
	T:=H^{[d]\setminus b}_{[d-L_b];d-L_b+1,d-L_b+2} \text{vec}(p^b({q^b})^T),\ \hat T := f_{[d]\setminus b;b}s^2 
\end{equation}
respectively, as demonstrated in Figure~\ref{figure:Tdef}, where $p^b, q^b$ appear in Corollary
\ref{corollary:two TT}.  Corollary \ref{corollary:two TT} implies
\begin{equation}
  \label{E def}
  T = \hat T + E,\quad \frac{\|E\|_F^2}{\| \hat T \|_F^2} \leq \kappa^2 (1-\frac{\alpha}{\kappa^4}).
\end{equation}
In the following, we want to show that we can approximately extract the $H^k$'s in region $a$.  For
this, we need to take the right-inverses of $\tilde H^{c_1}_{L_{c_1}+1;[L_{c_1}]}$ and $\tilde
H^{c_2}_{L_{c_2}+1;[L_{c_2}]}$, defined in Figure~\ref{figure:Tdecom}. This requires a singular value
lower bound, provided by the next lemma.
\begin{figure}[!ht]
  \subfloat[\label{figure:Tdef}]{%
	\includegraphics[width = 0.5\columnwidth,trim = 1cm 0cm 3cm 0cm,clip]{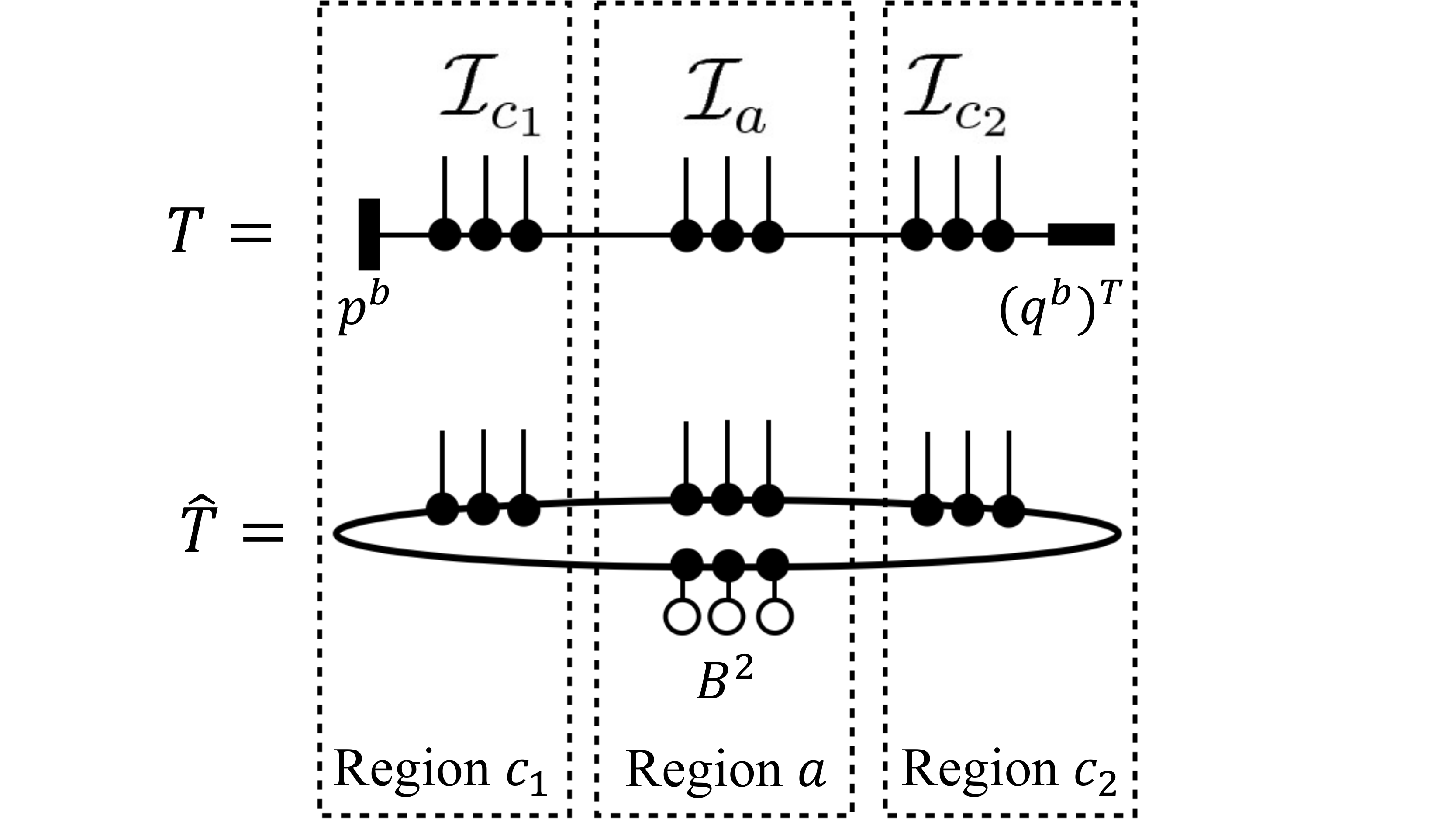}
  }
  \hfill
  \subfloat[\label{figure:Tdecom}]{%
	\includegraphics[width = 0.5\columnwidth,trim = 0cm 2cm 3cm 2cm,clip]{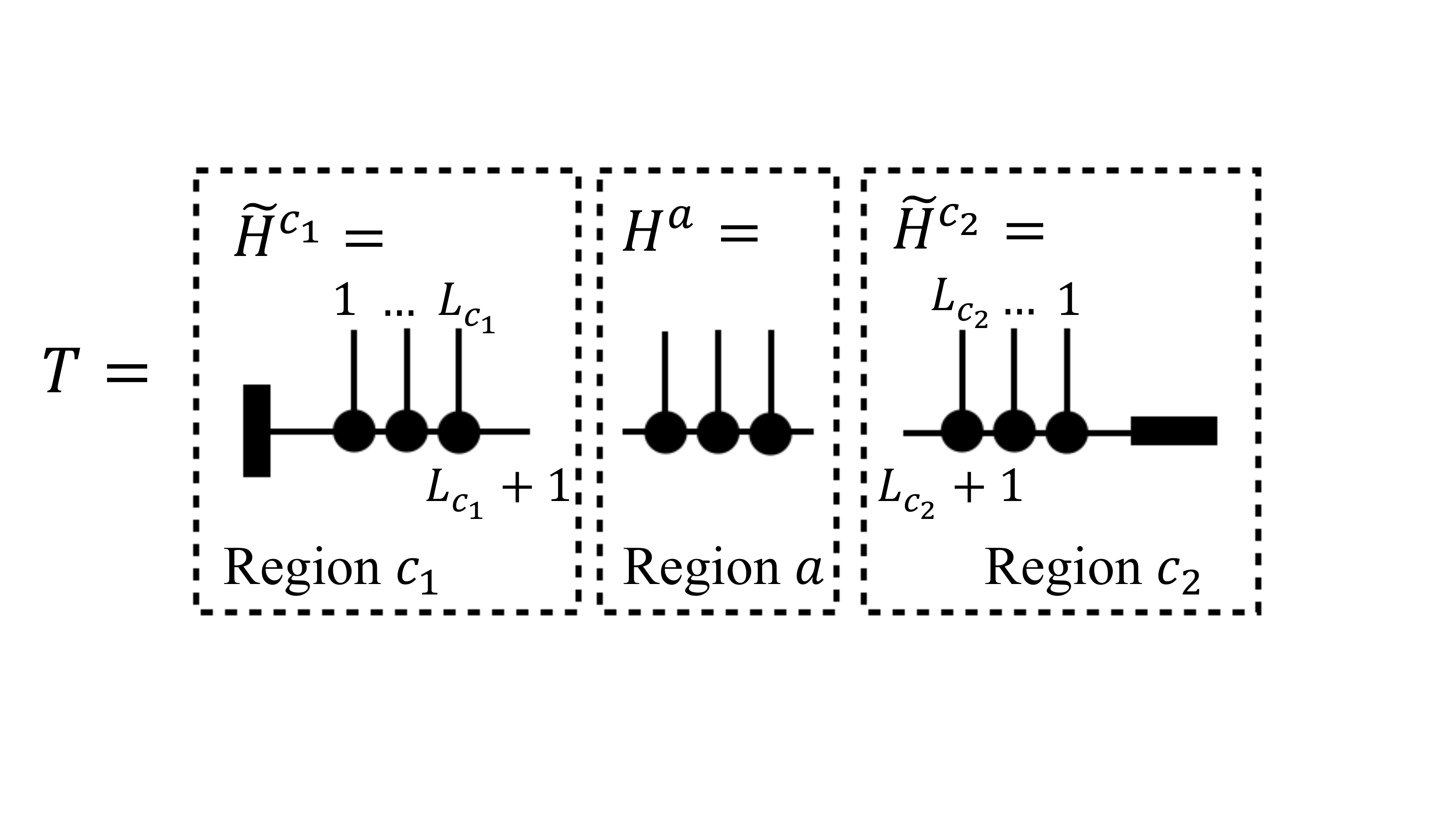}
  }
  \caption{(a) Definition of $T$ and $\hat T$. The dimensions in region $a,c_1,c_2$ are group into
    $\mathcal{I}_a,\mathcal{I}_{c_1},\mathcal{I}_{c_2}$ respectively for the tensors $T$ and $\hat
    T$. (b) Individual components of $T$.}
\end{figure}

\begin{lemma}\label{lemma:eig lower bound}
  Let $\sigma_k: \mathbb{R}^{m_1\times m_2}\rightarrow \mathbb{R}$ be a function that extracts
  the $k-th$ singular value of a $m_1\times m_2$ matrix. Then
  \begin{equation}
	\frac{\sigma_r(\tilde H^{c_1}_{L_{c_1}+1;[L_{c_1}]})^2 \sigma_r(\tilde
      H^{c_2}_{L_{c_2}+1;[L_{c_2}]})^2}{ \|\tilde H^{[d]\setminus a}_{d-L_a+1,d-L_a+2;[d-L_a]}
      \|_2^2}\geq \frac{1}{\kappa^6} - \frac{2 \sqrt{r}}{\kappa^2}\sqrt{ 1 -
      \frac{\alpha}{\kappa^4}}
  \end{equation}
  assuming
  \begin{equation}
    \frac{1}{\kappa^4} - 2 \sqrt{r}\sqrt{ 1 - \frac{\alpha}{\kappa^4}}\geq 0.
  \end{equation}
\end{lemma}
\begin{proof}
  Firstly,
  \begin{align}
	\frac{\sigma_{r^2}(T_{\mathcal{I}_{a}; \mathcal{I}_{c_1} \mathcal{I}_{c_2}})^2}{\|\hat T_{\mathcal{I}_{a}; \mathcal{I}_{c_1},\mathcal{I}_{c_2}}\|_2^2} &\leq \frac{\|H^a_{[L_a];L_a+1,L_a+2}\|_2^2 \sigma_{r^2}(\tilde H^{c_1}_{L_{c_1}+1;[L_{c_1}]} \otimes \tilde H^{c_2}_{L_{c_2}+1;[L_{c_2}]})^2}{\|\hat T_{\mathcal{I}_{a}; \mathcal{I}_{c_1},\mathcal{I}_{c_2}}\|_2^2}\cr
	&= \frac{\|H^a_{[L_a];L_a+1,L_a+2}\|_2^2 \sigma_r(\tilde H^{c_1}_{L_{c_1}+1;[L_{c_1}]})^2 \sigma_r(\tilde H^{c_2}_{L_{c_2}+1;[L_{c_2}]})^2}{\| H^a_{[L_a];L_a+1,L_a+2} \tilde H^{[d]\setminus a}_{d-L_a+1,d-L_a+2;[d-L_a]}\|_2^2}\cr
	&\leq \frac{\|H^a_{[L_a];L_a+1,L_a+2}\|_2^2 \sigma_r(\tilde H^{c_1}_{L_{c_1}+1;[L_{c_1}]})^2 \sigma_r(\tilde H^{c_2}_{L_{c_2}+1;[L_{c_2}]})^2}{\sigma_{r^2} (H^a_{[L_a];L_a+1,L_a+2} )^2 \|\tilde H^{[d]\setminus a}_{d-L_a+1,d-L_a+2;[d-L_a]} \|_2^2}\cr
	&\leq \kappa^2 \frac{\sigma_r(\tilde H^{c_1}_{L_{c_1}+1;[L_{c_1}]})^2 \sigma_r(\tilde H^{c_2}_{L_{c_2}+1;[L_{c_2}]})^2}{ \|\tilde H^{[d]\setminus a}_{d-L_a+1,d-L_a+2;[d-L_a]} \|_2^2} \, .
  \end{align}
  The equality follows from $$\hat T_{\mathcal{I}_{a}; \mathcal{I}_{c_1},\mathcal{I}_{c_2}} =
  H^a_{[L_a];L_a+1,L_a+2} \tilde H^{[d]\setminus a}_{d-L_a+1,d-L_a+2;[d-L_a]},$$ which follows from
  \eqref{Tdef}, and the definition of $\tilde H^{[d]\setminus a}$ in Figure~\ref{figure:Hdef}.
  
  Observe that
  \begin{align}\label{eig bound}
	\frac{\sigma_{r^2}(T_{\mathcal{I}_{a}; \mathcal{I}_{c_1},\mathcal{I}_{c_2}})^2}{\|\hat T_{\mathcal{I}_{a}; \mathcal{I}_{c_1},\mathcal{I}_{c_2}}\|_2^2} &\geq \frac{\sigma_{r^2}(\hat T_{\mathcal{I}_{a}; \mathcal{I}_{c_1},\mathcal{I}_{c_2}})^2-2\| E \|_F\sigma_{r^2}(\hat T_{\mathcal{I}_{a}; \mathcal{I}_{c_1},\mathcal{I}_{c_2}}) +\| E \|_F^2}{\|\hat T_{\mathcal{I}_{a}; \mathcal{I}_{c_1},\mathcal{I}_{c_2}}\|_2^2} \cr
	&\geq \frac{\sigma_{r^2}(\hat T_{\mathcal{I}_{a}; \mathcal{I}_{c_1},\mathcal{I}_{c_2}})^2}{\|\hat T_{\mathcal{I}_{a}; \mathcal{I}_{c_1},\mathcal{I}_{c_2}}\|_2^2} -  \frac{2\| E \|_F\sigma_{r^2}(\hat T_{\mathcal{I}_{a}; \mathcal{I}_{c_1},\mathcal{I}_{c_2}}) }{\|\hat T_{\mathcal{I}_{a}; \mathcal{I}_{c_1},\mathcal{I}_{c_2}}\|_2^2}\cr
	&\geq  \frac{\sigma_{r^2}(H^a_{[L_a];L_a+1,L_a+2})^2 \sigma_{r^2}(\tilde H^{[d]\setminus a}_{d-L_a+1,d-L_a+2;[d-L_a]} )^2}{\| H^a_{[L_a];L_a+1,L_a+2}\|_2^2 \|\tilde H^{[d]\setminus a}_{d-L_a+1,d-L_a+2;[d-L_a]} \|_2^2 }\cr
	& \qquad - \frac{2\| E \|_F\sigma_{r^2}(\hat T_{\mathcal{I}_{a}; \mathcal{I}_{c_1},\mathcal{I}_{c_2}}) }{\|\hat T_{\mathcal{I}_{a}; \mathcal{I}_{c_1},\mathcal{I}_{c_2}}\|_2^2} \cr
	&\geq  \frac{1}{\kappa^4} - \frac{2\| E \|_F\sigma_{r^2}(\hat T_{\mathcal{I}_{a}; \mathcal{I}_{c_1},\mathcal{I}_{c_2}}) }{\|\hat T_{\mathcal{I}_{a}; \mathcal{I}_{c_1},\mathcal{I}_{c_2}}\|_2^2} \cr
	&\geq \frac{1}{\kappa^4}  - \frac{2\sqrt{r}\sigma_{r^2}(\hat T_{\mathcal{I}_{a}; \mathcal{I}_{c_1},\mathcal{I}_{c_2}}) \|E\|_F }{\|\hat T_{\mathcal{I}_{a} ; \mathcal{I}_{c_1},\mathcal{I}_{c_2}}\|_2 \|\hat T_{\mathcal{I}_{a} ; \mathcal{I}_{c_1},\mathcal{I}_{c_2}}\|_F} \cr
	&\geq  \frac{1}{\kappa^4}  - 2\sqrt{r}  \sqrt{ 1 - \frac{\alpha}{\kappa^4}},
  \end{align}
  we established the claim. The first inequality regarding perturbation of singular values follows
  from theorem by Mirsky \cite{mirsky1960symmetric}:
  \begin{equation}
    \bigl\lvert \sigma_{r^2}(T_{\mathcal{I}_{a}; \mathcal{I}_{c_1},\mathcal{I}_{c_2}}) - 	\sigma_{r^2}(\hat T_{\mathcal{I}_{a}; \mathcal{I}_{c_1},\mathcal{I}_{c_2}}) \bigr\rvert \leq \|E\|_2  \leq \|E\|_F,
  \end{equation}
  and assuming $\|E\|_F\leq \sigma_{r^2}(\hat T_{\mathcal{I}_{a};
    \mathcal{I}_{c_1},\mathcal{I}_{c_2}})$ . Such assumption holds when demanding the lower bound in
  \eqref{eig bound} to be nonnegative, i.e.
  \begin{equation}
	\frac{\sigma_{r^2}(\hat T_{\mathcal{I}_{a}; \mathcal{I}_{c_1},\mathcal{I}_{c_2}})^2}{\|\hat T_{\mathcal{I}_{a}; \mathcal{I}_{c_1},\mathcal{I}_{c_2}}\|_2^2} -  \frac{2\| E \|_F\sigma_{r^2}(\hat T_{\mathcal{I}_{a}; \mathcal{I}_{c_1},\mathcal{I}_{c_2}}) }{\|\hat T_{\mathcal{I}_{a}; \mathcal{I}_{c_1},\mathcal{I}_{c_2}}\|_2^2} \geq  \frac{1}{\kappa^4}  - 2\sqrt{r}  \sqrt{ 1 - \frac{\alpha}{\kappa^4}} \geq 0
  \end{equation}
  The last inequality follows from Corollary \ref{corollary:two TT}.
\end{proof}

In the next lemma, we prove that when applying Algorithm~\ref{alg:initialization} to $\hat T$, where
$\hat T$ is treated as a 3-tensor formed from grouping the dimensions in each of set
$\mathcal{I}_a,\mathcal{I}_{c_1}\mathcal{I}_{c_2}$, gives close approximation to $\hat T$.
\begin{lemma}\label{lemma:TT approx}
  Let
  \begin{align}
	\Pi_1 &:= \bigl\{Y\ \vert \ Y=XX^T, X \in \mathbb{R}^{n^{L_{c_1}}\times r},\ X^T X = I \bigr\},\cr
	\Pi_2 &:= \bigl\{Y\ \vert \ Y=XX^T, X \in \mathbb{R}^{n^{L_{c_2}}\times r},\ X^T X = I \bigr\},
  \end{align}
  where $I$ is the identity matrix. Let $P_1^*\in \Pi_1$ be the best rank-$r$ projection for $\hat T_{\mathcal{I}_{c_2}\mathcal{I}_a;\mathcal{I}_{c_1}}$ such that $\hat T_{\mathcal{I}_{c_2}\mathcal{I}_a;\mathcal{I}_{c_1}}P_1^* \approx \hat T_{\mathcal{I}_{c_2}\mathcal{I}_a;\mathcal{I}_{c_1}}$ in Frobenius-norm, and  $$P_2^* = \min_{P_2\in \Pi_2}\| (\hat T_{\mathcal{I}_a;\mathcal{I}_{c_1}\mathcal{I}_{c_2}}(I\otimes P_2) - \hat T_{\mathcal{I}_a;\mathcal{I}_{c_1}\mathcal{I}_{c_2}})(P_1^*\otimes I)\|_F^2.$$
  Then
  \begin{equation}
	\| \hat T_{\mathcal{I}_a;\mathcal{I}_{c_1}\mathcal{I}_{c_2}}(I\otimes P_2^*)(P_1^*\otimes I)
    - \hat T_{\mathcal{I}_a;\mathcal{I}_{c_1}\mathcal{I}_{c_2}}\|_F^2 \leq 2\|E\|_F^2.
  \end{equation}
\end{lemma}
\begin{proof}
  To simplify the notations, let $\tilde T_{\mathcal{I}_a;\mathcal{I}_{c_1}\mathcal{I}_{c_2}} :=
  \hat T_{\mathcal{I}_a;\mathcal{I}_{c_1}\mathcal{I}_{c_2}}(I\otimes P_2)$.  Then
  \begin{align}
	& \min_{P_2\in \Pi_2}\| \hat T_{\mathcal{I}_a;\mathcal{I}_{c_1}\mathcal{I}_{c_2}}(I\otimes P_2)(P_1^*\otimes I) - \hat T_{\mathcal{I}_a;\mathcal{I}_{c_1}\mathcal{I}_{c_2}}\|_F^2\cr
	=& \min_{P_2\in \Pi_2}\| (\tilde T_{\mathcal{I}_a;\mathcal{I}_{c_1}\mathcal{I}_{c_2}} -  \hat T_{\mathcal{I}_a;\mathcal{I}_{c_1}\mathcal{I}_{c_2}} +  \hat T_{\mathcal{I}_a;\mathcal{I}_{c_1}\mathcal{I}_{c_2}}) (P_1^* \otimes I) -  \hat T_{\mathcal{I}_a;\mathcal{I}_{c_1}\mathcal{I}_{c_2}}\|_F^2\cr
	=&\min_{P_2\in \Pi_2}\| \hat T_{\mathcal{I}_a;\mathcal{I}_{c_1}\mathcal{I}_{c_2}}(I-P_1^*\otimes I)\|_F^2 + \| (\tilde T_{\mathcal{I}_a;\mathcal{I}_{c_1}\mathcal{I}_{c_2}} - \hat T_{\mathcal{I}_a;\mathcal{I}_{c_1}\mathcal{I}_{c_2}})(P_1^*\otimes I)\|_F^2\cr
	\leq&\min_{P_2\in \Pi_2}\| \hat T_{\mathcal{I}_a;\mathcal{I}_{c_1}\mathcal{I}_{c_2}}(I-P_1^*\otimes I)\|_F^2 + \| \tilde T_{\mathcal{I}_a;\mathcal{I}_{c_1}\mathcal{I}_{c_2}} - \hat T_{\mathcal{I}_a;\mathcal{I}_{c_1}\mathcal{I}_{c_2}}\|_F^2\cr
	=& \min_{P_2\in \Pi_2}\| \hat T_{\mathcal{I}_a;\mathcal{I}_{c_1}\mathcal{I}_{c_2}}(I-P_1^*\otimes I)\|_F^2 + \| \hat T_{\mathcal{I}_a;\mathcal{I}_{c_1}\mathcal{I}_{c_2}}(I-I\otimes P_2)\|_F^2.
  \end{align}
  The inequality comes from the fact that $P_1^*\otimes I$ is a projection matrix. Next,
  \begin{multline}
	\| \hat T_{\mathcal{I}_a;\mathcal{I}_{c_1}\mathcal{I}_{c_2}}(I-P_1^*\otimes I)\|_F^2 + \min_{P_2\in \Pi_2}\| \hat T_{\mathcal{I}_a;\mathcal{I}_{c_1}\mathcal{I}_{c_2}}(I-I\otimes P_2)\|_F^2 \\
	= \min_{P_1\in \Pi_1} \| \hat T_{\mathcal{I}_a;\mathcal{I}_{c_1}\mathcal{I}_{c_2}}(I-P_1\otimes I)\|_F^2 + \min_{P_2\in \Pi_2}\| \hat T_{\mathcal{I}_a;\mathcal{I}_{c_1}\mathcal{I}_{c_2}}(I-I\otimes P_2)\|_F^2\\
	\leq \|E\|_F^2 + \|E\|_F^2 \leq 2\|E\|_F^2,
  \end{multline}
  and we can conclude the lemma. The equality comes from the definition of $P_1^*$, whereas the
  inequality is due to the facts that $P_1, P_2$ are rank-$r$ projectors, and there exists $T$ such
  that $ \hat T = T - E$ where $\text{rank}(T_{\mathcal{I}_{c_1}\mathcal{I}_a;\mathcal{I}_{c_2}})$,
  $\text{rank}(T_{\mathcal{I}_{c_1};\mathcal{I}_a\mathcal{I}_{c_2}})$ $\leq r$.
\end{proof}

We are ready to state the final proposition.
\begin{proposition}\label{proposition:stability}
  Let 
  \begin{eqnarray}
	\doublehat{T}_{\mathcal{I}_a;\mathcal{I}_{c_1}\mathcal{I}_{c_2}}:= \hat T_{\mathcal{I}_a;\mathcal{I}_{c_1}\mathcal{I}_{c_2}}(I\otimes P_2^*)(P_1^*\otimes I) 
  \end{eqnarray}
  where $P_1^*, P_2^*$ are defined in Lemma~\ref{lemma:TT approx}. Then 
  \begin{equation}
	\frac{\| H^a_{[L_a];L_a+1,L_a+2} - \doublehat{T}_{\mathcal{I}_{a}; \mathcal{I}_{c_1} \mathcal{I}_{c_2}}(\tilde H^{c_1}_{L_{c_1}+1;[L_{c_1}]}\otimes \tilde H^{c_2}_{L_{c_2}+1;[L_{c_2}]})^{\dagger} \|_F^2}{\| H^a_{[L_a];L_a+1,L_a+2}\|_F^2 }\leq \frac{(1+\sqrt{2})^2\kappa^4(1-\frac{\alpha}{\kappa^4})}{\frac{1}{\kappa^4} - 2\sqrt{r}\sqrt{ 1 - \frac{\alpha}{\kappa^4}}},
  \end{equation}
  where ``$\dagger$'' is used to denote the pseudo-inverse of a matrix, if the upper bound is positive. When $\kappa = 1+\delta \kappa$ and $\alpha = 1-\delta \alpha$ where $\delta \kappa, \delta \alpha\geq 0$ are small parameters, we have
  \begin{equation}
	\frac{\| H^a_{[L_a];L_a+1,L_a+2} - \doublehat{T}_{\mathcal{I}_{a}; \mathcal{I}_{c_1} \mathcal{I}_{c_2}}(\tilde H^{c_1}_{L_{c_1}+1;[L_{c_1}]}\otimes \tilde H^{c_2}_{L_{c_2}+1;[L_{c_2}]})^{\dagger} \|_F^2}{\| H^a_{[L_a];L_a+1,L_a+2}\|_F^2 } \leq O(\delta \alpha + 4 \delta \kappa ).
  \end{equation}
\end{proposition}
\begin{proof}
  From Lemma~\ref{lemma:TT approx} and \eqref{E def}, we get 
  \begin{align}
	&\|  \doublehat{T}_{\mathcal{I}_a;\mathcal{I}_{c_1}\mathcal{I}_{c_2}}- T_{\mathcal{I}_a;\mathcal{I}_{c_1}\mathcal{I}_{c_2}} \|_F \cr
	=& \| \hat T_{\mathcal{I}_a;\mathcal{I}_{c_1}\mathcal{I}_{c_2}}(I\otimes P_2^*)(P_1^*\otimes I) - T_{\mathcal{I}_a;\mathcal{I}_{c_1}\mathcal{I}_{c_2}} \|_F \cr
	\leq& \| \hat T_{\mathcal{I}_a;\mathcal{I}_{c_1}\mathcal{I}_{c_2}}(I\otimes P_2^*)(P_1^*\otimes I) - \hat T_{\mathcal{I}_a;\mathcal{I}_{c_1}\mathcal{I}_{c_2}}\|_F + \| \hat T_{\mathcal{I}_a;\mathcal{I}_{c_1}\mathcal{I}_{c_2}} - T_{\mathcal{I}_a;\mathcal{I}_{c_1}\mathcal{I}_{c_2}}\|_F\cr
	\leq& (1+\sqrt{2})\|E\|_F. \label{TR to TT}
  \end{align}
  Recall that
  \begin{equation}
	\label{inverse T}
	H^a_{[L_a];L_a+1,L_a+2} = T_{\mathcal{I}_{a}; \mathcal{I}_{c_1},\mathcal{I}_{c_2}}(\tilde H^{c_1}_{L_{c_1}+1;[L_{c_1}]}\otimes \tilde H^{c_2}_{L_{c_2}+1;[L_{c_2}]})^{\dagger},
  \end{equation}
  where the existence of a full-rank pseudo-inverse is guaranteed by the singular value lower bound
  in Lemma~\ref{lemma:eig lower bound}, we have
  \begin{align}
	&\frac{\|H^a_{[L_a];L_a+1,L_a+2} - \doublehat{T}_{\mathcal{I}_{a}; \mathcal{I}_{c_1},\mathcal{I}_{c_2}}(\tilde H^{c_1}_{L_{c_1}+1;[L_{c_1}]}\otimes \tilde H^{c_2}_{L_{c_2}+1;[L_{c_2}]})^{\dagger}\|_F^2}{\|H^a_{[L_a];L_a+1,L_a+2} \|_F^2}\cr
	\leq& \frac{(1+\sqrt{2})^2\|E\|_F^2\|(\tilde H^{c_1}_{L_{c_1}+1;[L_{c_1}]}\otimes \tilde H^{c_2}_{L_{c_2}+1;[L_{c_2}]})^{\dagger}\|_2^2 }{ \|H^a_{[L_a];L_a+1,L_a+2} \|_F^2}\cr
	\leq& \frac{(1+\sqrt{2})^2\|E\|_F^2}{\sigma_r(\tilde H^{c_1}_{L_{c_1}+1;[L_{c_1}]})^2 \sigma_r(\tilde H^{c_2}_{L_{c_2}+1;[L_{c_2}]})^2 \|H^a_{[L_a];L_a+1,L_a+2} \|_F^2}\cr
	=&\frac{(1+\sqrt{2})^2\|\hat T \|_F^2}{\sigma_r(\tilde H^{c_1}_{L_{c_1}+1;[L_{c_1}]})^2 \sigma_r(\tilde H^{c_2}_{L_{c_2}+1;[L_{c_2}]})^2 \|H^a_{[L_a];L_a+1,L_a+2} \|_F^2}\frac{\|E\|_F^2}{\|\hat T\|_F^2}\cr
	\leq&\frac{(1+\sqrt{2})^2\|\tilde H^{[d]\setminus a}_{d-L_a+1,d-L_a+2;[d-L_a]} \|_2^2}{\sigma_r(\tilde H^{c_1}_{L_{c_1}+1;[L_{c_1}]})^2 \sigma_r(\tilde H^{c_2}_{L_{c_2}+1;[L_{c_2}]})^2}  \frac{\|E\|_F^2}{\|\hat T\|_F^2}\cr
	\leq& \frac{(1+\sqrt{2})^2}{\frac{1}{\kappa^6} - \frac{2 \sqrt{r}}{\kappa^2}\sqrt{ 1 - \frac{\alpha}{\kappa^4}}} \kappa^2(1-\frac{\alpha}{\kappa^4}).
  \end{align}
  The first inequality follows from \eqref{TR to TT} and \eqref{inverse T}, and the last inequality
  follows from Corollary~\ref{corollary:two TT} and Lemma~\ref{lemma:eig lower bound}.
\end{proof}
When $L_a=L_{c_1} = L_{c_2}=1$, applying Algorithm~\ref{alg:initialization} to $\hat T$ results
$\doublehat{T}$ (represented by the tensors $T^{a,L}, T^{a,C}$ and $T^{a,R}$). Therefore, this
proposition essentially implies $T^{a,C}$ approximates $H^a$ up to gauge transformation.

\bibliographystyle{amsplain}
\bibliography{bibref}

\end{document}